\documentclass{amsart}
\usepackage{amssymb}
\usepackage{amsrefs} 

\newtheorem{theorem}{Theorem}[section]

\newtheorem{corollary}[theorem]{Corollary}
\newtheorem{lemma}[theorem]{Lemma}
\newtheorem{question}[theorem]{Question}
\newtheorem{construction}[theorem]{Construction}
\theoremstyle{definition}
\newtheorem{example}[theorem]{Example}

\theoremstyle{remark}
\newtheorem*{remark}{Remark}

\DeclareMathOperator{\Aut}{Aut}
\DeclareMathOperator{\Br}{Br}
\DeclareMathOperator{\characteristic}{char}
\DeclareMathOperator{\Gal}{Gal}
\DeclareMathOperator{\Hom}{Hom}
\DeclareMathOperator{\Isom}{Isom}
\DeclareMathOperator{\Jac}{Jac}
\DeclareMathOperator{\PGL}{PGL}
\DeclareMathOperator{\SL}{SL}
\DeclareMathOperator{\Spec}{Spec}

\DeclareMathOperator{\Stab}{Stab}
\DeclareMathOperator{\Tr}{Tr}
\DeclareMathOperator{\Tw}{Tw}

\newcommand{\indexin}[2]{[#1\,{:}\,#2]}
\newcommand{\col}{\,{:}\,}

\newcommand{\BC}{{\mathbb{C}}}
\newcommand{\F}{{\mathbb{F}}}

\newcommand{\PP}{{\mathbb{P}}}
\newcommand{\Q}{{\mathbb{Q}}}
\newcommand{\BR}{{\mathbb{R}}}
\newcommand{\Z}{{\mathbb{Z}}}

\newcommand{\no}{\hbox to 5em{\hfil\textup{no}\hfil}}
\newcommand{\yes}{\hbox to 5em{\hfil\textup{yes}\hfil}}

\renewcommand{\hat}{\widehat}
\renewcommand{\tilde}{\widetilde}

\newcommand\lowtilde{\lower0.7ex\hbox{~}}

\newcommand{\mybar}[1]{#1\llap{$\overline{\phantom{\rm#1}}$}}
\renewcommand{\bar}{\mybar}

\newcommand{\Kb}{\bar{K}}
\newcommand{\Gb}{\bar{G}}

\newcommand{\Fbar}{\overline{\mathbb{F}}}
\newcommand{\Kbar}{{\bar{\scriptstyle K}}}

\begin{document}

\title[Nonisomorphic curves]
{Nonisomorphic curves that become isomorphic
over extensions of coprime degrees}

\author[Goldstein]{Daniel Goldstein}           
\address{Center for Communications Research,
         4320 Westerra Court,
         San Diego, CA 92121-1967, USA.}
\email{dgoldste@ccrwest.org}

\author[Guralnick]{\hbox{Robert M.~Guralnick}}    
\address{Department of Mathematics, 
         University of Southern California, 
         Los Angeles, CA 90089-2532, USA.}
\email{guralnic@usc.edu}

\author[Howe]{\hbox{Everett W.~Howe}}
\address{Center for Communications Research,
         4320 Westerra Court,
         San Diego, CA 92121-1967, USA.}
\email{however@alumni.caltech.edu}
\urladdr{http://www.alumni.caltech.edu/\lowtilde{}however/}

\author[Zieve]{\hbox{Michael E.~Zieve}}
\address{Center for Communications Research,
         805 Bunn Drive,
         Princeton, NJ 08540-1966, USA.}
\email{zieve@math.rutgers.edu}
\urladdr{http://www.math.rutgers.edu/\lowtilde{}zieve/}

\date{7 May 2008}


\begin{abstract}
We show that one can find two nonisomorphic
curves over a field $K$ that become isomorphic
to one another over two finite extensions of $K$
whose degrees over $K$ are coprime to one another.

More specifically, 
let $K_0$ be an arbitrary prime field and let
$r>1$ and $s>1$ be integers that are coprime 
to one another.  We show that one can find a 
finite extension $K$ of $K_0$,  a degree-$r$ 
extension $L$ of~$K$, a degree-$s$ extension 
$M$ of~$K$, and two curves $C$ and $D$ over $K$
such that $C$ and $D$ become isomorphic to one
another over $L$ and over~$M$, but not over
any proper subextensions of $L/K$ or~$M/K$.

We show that such $C$ and $D$ can never have 
genus~$0$, and that if $K$ is finite, $C$ 
and $D$ can have genus $1$ if and only if 
$\{r,s\} = \{2,3\}$ and $K$ is an odd-degree
extension of~$\F_3$.  On the other hand, when
$\{r,s\}=\{2,3\}$ we show that genus-$2$ examples
occur in every characteristic other than~$3$.

Our detailed analysis of the case $\{r,s\} = \{2,3\}$
shows that over every finite field $K$
there exist nonisomorphic curves $C$ and $D$
that become isomorphic to one another over the
quadratic and cubic extensions of~$K$.

Most of our proofs rely on Galois cohomology.
Without using Galois cohomology, we show that
two nonisomorphic genus-$0$ curves over an
arbitrary field remain nonisomorphic over 
every odd-degree extension of the base field.
\end{abstract}

\maketitle

\section{Introduction}
\label{S:intro}
Suppose $C$ is a curve over a field $K$, and suppose $L$ is an extension
field of~$K$.  An \emph{$L$-twist} of $C$ is a curve $D$ over $K$ that
becomes isomorphic to $C$ when the base field is extended to~$L$,
and a \emph{twist} of $C$ is a curve $D$ over $K$ that becomes isomorphic
to $C$ over \emph{some} field extension.
The simplest version of the question we address in this paper is the following:
\begin{question}
\label{Q:oldmain}
Suppose $C$ is a curve over a field $K$, and suppose $L$ and $M$ are
finite extensions of $K$ whose degrees over $K$ are coprime to one
another. If $D$ is a curve over $K$ that is simultaneously an
$L$-twist of $C$ and an $M$-twist of $C$, must $D$ be isomorphic to
$C$ over~$K$\/\textup{?}
\end{question}
There are a number of published papers that consider a generalization of 
Question~\ref{Q:oldmain} (see the discussion in~\S\ref{S:related}),
but we have not found any papers specifically addressing
Question~\ref{Q:oldmain} itself.  We will show that the answer
to the question is `no'.
   
We actually answer a more refined question.
Suppose $D$ is a twist of $C$.  We say that an extension $L$ of $K$ is a 
\emph{minimal isomorphism extension}
for $C$ and $D$ if $C$ and $D$ become isomorphic to one another over
$L$ but not over any proper subextension of~$L/K$.
\begin{question}
\label{Q:main}
Suppose $r>1$ and $s>1$ are integers that are coprime to one another.
Does there exist a curve $C$ over a field $K$, a twist $D$ of $C$,
and a pair of extensions $L$ and $M$ of $K$ of degrees $r$ and $s$, 
respectively, such that $L$ and $M$ are both minimal isomorphism
extensions for $C$ and~$D$\/\textup{?}
\end{question}

The answer to Question~\ref{Q:main} is `yes', as the following theorem shows.

\begin{theorem}
\label{T:main}
Let $K_0$ be a prime field and let $r>1$ and $s>1$ be integers whose 
greatest common divisor is~$1$.  Then there exist curves $C$ and $D$
over a finite extension $K$ of $K_0$
that are twists of one another and that have
minimal isomorphism extensions of degrees $r$ and $s$ over $K$.
\end{theorem}

In Section~\ref{S:proofs} we provide two proofs of this theorem for the 
special case where $K_0$ is finite.  The second proof shows
that we may take $K$ to be any even-degree extension of~$K_0$.
We can also write down very explicit examples of such curves $C$ and $D$
in many special cases.  In particular, 
when $K_0=\Q$ we can write down explicit examples for every $r$ and $s$, thus
completing the proof of Theorem~\ref{T:main}.  
We note that when $r$ and $s$ are both prime, we can write down explicit
examples over fields of every characteristic.

We also show that over finite fields, curves satisfying the conclusion of
Theorem~\ref{T:main} must have large geometric automorphism groups.
 
\begin{theorem}
\label{T:automorphism}
Let $K_0$, $r$, $s$, $C$, and $D$ be as in Theorem~\textup{\ref{T:main}},
and suppose that $K_0$ is finite.  Then the geometric automorphism 
group of $C$ \textup{(}and of $D$\textup{)}  contains a finite subgroup of
order divisible by $rs$, but not equal to $rs$.  Furthermore, if
$rs \equiv 2\bmod 4$ then this automorphism group contains a finite 
subgroup whose order is divisible by~$2rs$.
\end{theorem}

We actually prove a slightly stronger statement; see
Theorem~\ref{T:automorphism-plus-galois}. 

In addition, we take a closer look at the possibilities 
when $C$ and $D$ have small genus.   When $K_0$ is finite,
the curves that occur in Theorem~\ref{T:main} can never have genus~$0$, 
because over a finite field all twists of a genus-$0$ curve are trivial.
But even over an arbitrary field, genus-$0$ examples do not exist.

\begin{theorem}
\label{T:genus0}
Suppose $C$ and $D$ are curves of genus $0$ over a field $K$ that become
isomorphic to one another over an odd-degree extension of $K$.  Then
$C\cong D$.
\end{theorem}

(Note that this theorem would be false without the restriction to odd-degree
extensions; the genus-$0$ curve $x^2 + y^2 = -1$ over $\BR$ is a 
nontrivial quadratic twist of the projective line over~$\BR$.)

\begin{corollary}
\label{C:genus0}
For every $r$ and $s$, the answer to Question~\textup{{\ref{Q:main}}} is `no' 
if the curve $C$ is required to have genus~$0$.
\end{corollary}

We also show that over finite fields examples of genus-$1$ curves as in 
Theorem~\ref{T:main} occur only in a very special case.

\begin{theorem}
\label{T:genus1}
Suppose that $C$ and $D$ are nonisomorphic
curves of genus $1$ over a finite field $K$ such that $C$ and $D$ have
minimal isomorphism extensions of degrees $r$ and $s$, where $r$ and $s$ are
coprime to one another.  Then $\{r,s\} = \{2,3\}$,
the field $K$ is an odd-degree extension of\/ $\F_3$, 
and $C$ and $D$ have supersingular Jacobians.
Conversely, for every odd-degree extension $K$ of\/ $\F_3$ there
are genus-$1$ curves $C$ and $D$ over $K$  that have
minimal isomorphism extensions of degrees $2$ and~$3$.
\end{theorem}

On the other hand, when $\{r,s\} = \{2,3\}$ we can get genus-$2$ examples 
over every finite field whose characteristic is not~$3$.

\begin{theorem}
\label{T:genus2}
If $K$ is a finite field of characteristic not~$3$, then there 
exist genus-$2$ curves $C$ and $D$ over $K$
that have minimal isomorphism extensions of degrees $2$ and~$3$.
If $K$ is a finite field of characteristic~$3$, then no such
curves exist over $K$.
\end{theorem}

Combining Theorems~\ref{T:genus1} and~\ref{T:genus2} and the second
proof of Theorem~\ref{T:main} for finite fields, we obtain an interesting corollary:

\begin{corollary}
\label{C:allfield}
Over every finite field $K$, there exist nonisomorphic curves $C$ and 
$D$ that become isomorphic to one another over the quadratic and cubic 
extensions of $K$.
\end{corollary}

When $K_0$ is finite and $\{r,s\}\neq\{2,3\}$, examples of genus-$2$ curves as
in Theorem~\ref{T:main} are very special.

\begin{theorem}
\label{T:genus2bigger}
Let $r>1$ and $s>1$ be coprime integers with $\{r,s\} \neq \{2,3\}$.
Suppose that $C$ and $D$ are nonisomorphic curves
of genus $2$ over a finite field $K$ such that $C$ and $D$ have minimal
isomorphism extensions of degrees $r$ and $s$.
Then $\{r,s\} = \{2,5\}$, the field $K$ is an odd-degree extension 
of\/~$\F_5$, and there is an element $a$ of $K$ whose trace to\/ $\F_5$
is nonzero such that $C$ and $D$ are isomorphic to the curves
\begin{align*}
y^2 & = x^5 - x + a \text{\quad and}\\
y^2 & = x^5 - x + 2a,
\end{align*}
respectively.  Conversely, if $K$ is an odd-degree extension of\/ $\F_5$
and $a\in K$ has nonzero absolute trace, then the two curves given above
have minimal isomorphism extensions of degrees $2$ and $5$ over~$K$.
\end{theorem}

Our results lead naturally to a number of related questions.

\begin{question}
\label{Q:manyCs}
Given coprime integers $r>1$ and $s>1$, is there an upper bound on
the size of a set of curves $\{C_i\}$ over a field $K$
such that each pair of curves $(C_i,C_j)$ with $i\neq j$ has minimal
isomorphism extensions of degrees $r$ and $s$\/\textup{?}
\end{question}

\begin{question}
\label{Q:manyrs}
Given a finite set $\{r_i\}$ of integers greater than $1$, no one of
which divides any of the others, do there exist curves $C$ and $D$
over a field $K$ that have minimal isomorphism extensions of degree $r_i$
for each~$i$\/\textup{?}
\end{question}

Our methods can be used to show that the answer to
Question~\ref{Q:manyCs} is no and the answer to Question~\ref{Q:manyrs}
is yes; furthermore, the answers remain the same even if the field $K$ 
is required to be a finite field of a given positive characteristic.
As we will see, questions of this sort are related to the following
natural question:

\begin{question}
\label{Q:automorphismgroup}
Given a finite field $K$, a finite group $G$,
and an automorphism $\varphi$ of $G$,
does there exists a curve $C$ over $K$ whose
geometric automorphism group is isomorphic to $G$, 
with the isomorphism taking the action of Frobenius on 
the automorphism group to $\varphi$\/\textup{?}
\end{question}

This question is related to a result of Madden and 
Valentini~\cite{madden-valentini}, who show that for every finite
group $G$ and every field $K$, there is a curve over 
the algebraic closure of $K$ whose automorphism group is 
isomorphic to~$G$.

For finite fields $K$, specifying a finite extension $L$ of $K$ (up to 
isomorphism) is equivalent to specifying the degree of $L$ over $K$.
For arbitrary fields this is of course no longer the case.
This leads us to our final open question:

\begin{question}
\label{Q:arbitraryLM}
Given two linearly disjoint finite extension fields $L$ and $M$ 
of a field $K$, do there exist curves $C$ and $D$ over $K$ having
$L$ and $M$ as minimal isomorphism extensions\/\textup{?}
\end{question}

In Section~\ref{S:twists} we give some background information on
nonabelian Galois cohomology and twists of curves.  In
Section~\ref{S:proofs} we show how Theorem~\ref{T:main} can be proven
for finite fields if we can provide examples of curves with certain 
automorphism groups; we then complete the proof of Theorem~\ref{T:main}
for finite fields by constructing --- in two different ways --- curves 
with the right kind of automorphism group. In Section~\ref{S:proof3}
we provide several explicit constructions of curves satisfying the
conclusion of Theorem~\ref{T:main} for many different values of $r$ and $s$;
in particular, these examples give a complete proof of the theorem
in the case $K_0=\Q$.
In Section~\ref{S:genus0} we give a cohomology-free proof of
Theorem~\ref{T:genus0}.  In Section~\ref{S:grouptheory} we prove some
results from group theory that we require in the sections that follow.
We prove Theorem~\ref{T:genus1} in Section~\ref{S:genus1}
and Theorem~\ref{T:automorphism} in Section~\ref{S:automorphism}.
In Section~\ref{S:genus2} we prove Theorems~\ref{T:genus2} 
and~\ref{T:genus2bigger}.  In fact, we prove a stronger version of
Theorem~\ref{T:genus2} that gives more information about the 
automorphism groups of the genus-$2$ examples. 
In Section~\ref{S:related} we close the paper with a discussion
of questions related to our results.

\emph{Conventions.}
In this paper, by a \emph{curve over a field $K$} we always mean a 
connected complete geometrically nonsingular one-dimensional scheme over
$\Spec K$.  It follows by definition that all morphisms of curves over
$K$ are themselves defined over~$K$.  
When we present a curve via explicit equations, we mean the normalization
of the projective closure of the possibly-singular variety defined by those
equations.
If $C$ is a curve over $K$ and if
$L$ is an extension field of $K$, then the fiber product 
$C\times_{\Spec K} \Spec L$ is a curve over $L$ that we denote~$C_L$.
We can restate our definition of a twist of a curve using this notation:
If $C$ is a curve over a field $K$ and if $L$ is an extension field of
$K$, then an \emph{$L$-twist} of $C$ is a curve $D$ over $K$ such that
$D_L \cong C_L$.  We denote the algebraic closure of a field $K$ by $\Kb$.
The \emph{geometric automorphism group} of a curve $C$
over a field $K$ is the group $\Aut C_\Kbar$.

A \emph{proper divisor} of a positive integer $n$ is a positive
divisor of $n$ that is strictly less than $n$.  A \emph{proper subextension}
of a field extension $L/K$ is a subfield of $L$ that contains $K$, but that
is not equal to $L$ itself.

\emph{Acknowledgments.}
The authors are grateful to Brad Brock, George Glauberman, Murray Schacher, and 
Joseph Wetherell for helpful discussions.

\section{Twists and cohomology}
\label{S:twists}
In this section we review the relationship between twists of curves and
Galois cohomology.  To simplify the exposition, we limit our discussion
to the case of curves over finite fields.  Our source for the material in
this section is Serre's book~\cite{serre:GC}.  Nonabelian cohomology is
discussed in ~\cite{serre:GC}*{{\S}I.5}, and the cohomological
interpretation of twists is given 
in~\cite{serre:GC}*{{\S}III.1}.

Let $K$ be a finite field, let $X$ be a curve over $K$, and let $G = \Aut X_\Kbar$.
We begin by defining the pointed cohomology set  $H^1(\Gal \Kb/K,G)$.

We view $G$ as a topological group by giving it the discrete topology,
and we give the Galois group $\Gal \Kb/K$ the profinite topology.  A
\emph{cocycle} is a continuous map $\sigma \mapsto a_\sigma$ from
$\Gal \Kb/K$ to $G$ that satisfies the \emph{cocycle condition}
$a_{\sigma\tau} = a_\sigma a_\tau^\sigma$ for all $\sigma, \tau$ in
$\Gal \Kb/K$.  Since the absolute Galois group of a finite field is
freely generated as a profinite group by the Frobenius
automorphism~$\varphi$, a cocycle is completely determined by where it
sends $\varphi$.  Furthermore, for every $g\in G$ there
is a cocycle that sends $\varphi$ to~$g$.

Two cocycles $\sigma\mapsto a_\sigma$ and $\sigma\mapsto b_\sigma$ are
\emph{cohomologous} if there is an element $c$ of $G$ such that
$b_\sigma = c^{-1} a_\sigma c^\sigma$ for all $\sigma\in \Gal \Kb/K$.
This defines an equivalence relation on the cocycles, and the set of 
equivalence classes is denoted $H^1(\Gal \Kb/K,G)$.  We give 
$H^1(\Gal \Kb/K,G)$ the structure of a pointed set by distinguishing the 
class of the cocycle that sends all of $\Gal \Kb/K$ to the identity
of~$G$.

Let $\Tw(X)$ be the set of $K$-isomorphism classes of $\Kb$-twists
of~$X$.  We give $\Tw(X)$ the structure of a pointed set by 
distinguishing the isomorphism class of $X$ itself. As is shown 
in~\cite{serre:GC}*{{\S}III.1}, there is a bijection $\beta$
between the pointed sets $\Tw(X)$ and $H^1(\Gal \Kb/K,G)$, defined as 
follows:

Suppose $C$ is an $\Kb$-twist of $X$, and let $f$ be an isomorphism from
$X_\Kbar$ to~$C_\Kbar$.  Let $f^\varphi$ be the isomorphism $X_\Kbar\to C_\Kbar$ 
obtained from $f$ by replacing every coefficient in the equations
defining $f$ with its image under the Frobenius automorphism of~$\Kb/K$; 
this gives us an isomorphism from $X_\Kbar$ to $C_\Kbar$ because $X$ and $C$ are
defined over $K$.   We define $\beta(C)$ to be the cohomology class of
the cocycle that sends the Frobenius $\varphi\in \Gal \Kb/K$ to the
automorphism $f^{-1} f^\varphi$ of $X$.

Suppose $L$ is a finite extension of $K$.  Then there is a natural map 
from $\Tw(X)$ to $\Tw(X_L)$, defined by sending the class of an 
$\Kb$-twist $C$ of $X$ to the class of the $\Kb$-twist $C_L$ of~$X_L$.
This map is also easy to describe in terms of cohomology sets: A cocycle
representing a class of $H^1(\Gal \Kb/K,G)$ is a map $c:\Gal \Kb/K \to G$ 
that satisfies the cocycle condition, and we can define a map 
$H^1(\Gal \Kb/K,G) \to H^1(\Gal \Kb/L,G)$ by sending the class of a
cocycle $c$ to the class of $c|_{\Gal \Kbar/L}$.  We can also easily
describe this map in terms of Frobenius elements.  If $\varphi_K$ and
$\varphi_L$ are the Frobenius automorphisms of $\Kb/K$ and $\Kb/L$, 
respectively, then $\varphi_L = \varphi_K^n$, where
$n = \indexin{L}{K}$.  If a cocycle of $H^1(\Gal \Kb/K,G)$ sends 
$\varphi_K$ to $g$, then the corresponding cocycle of $H^1(\Gal \Kb/L,G)$
sends $\varphi_L$ to $g g^{\varphi_K} \cdots g^{\varphi_K^{n-1}}$.

A special case of this fact is important enough for our argument that we
state it as a lemma.

\begin{lemma}
\label{L:cohomology}
Let notation be as in the discussion above.  Suppose\/ $\Gal \Kb/K$ acts
trivially on $G = \Aut X_\Kbar$.  Then $H^1(\Gal \Kb/K,G)$ is isomorphic
\textup{(}as a pointed set\,\textup{)} to the set of conjugacy classes of~$G$.
If $C$ is a $\Kb$-twist of $X$ corresponding to the conjugacy class of
$g\in G$, then $C_L$ is the $\Kb$-twist of $X_L$ corresponding to the
conjugacy class of $g^n \in G$.
\end{lemma}

\begin{proof}
If $\Gal \Kb/K$ acts trivially on $G$ (that is, if all of the geometric
automorphisms of $X$ are defined over $K$), then the equivalence
relation for two cocycles being cohomologous degenerates into the
equivalence relation of conjugacy in~$G$, so $H^1(\Gal \Kb/K,G)$ is the
set of conjugacy classes of~$G$.

Suppose $C$ is a twist of $X$ corresponding to the cocycle that sends 
$\varphi_K$ to~$g$.  We have already noted that the cocycle representing
the class of the twist $C_L$ of $X_L$ is the cocycle that sends
$\varphi_L$ to $g g^{\varphi_K} \cdots g^{\varphi_K^{n-1}}$.  Since
$\varphi_K$ acts trivially on~$G$, this element is simply~$g^n$.  
\end{proof}

We close by noting that the relationship between twists of curves
and cohomology groups gives us a nice result about the automorphism
groups of twists of a given curve.  We state the result both for
curves and for \emph{pointed curves}, that is, curves with a marked
point.  Note that an elliptic curve is a pointed curve if we take 
the origin to be the marked point.

\begin{lemma}
\label{L:weight}
Let $X$ be a curve \textup{(}or a pointed curve\textup{)} over a finite field $K$,
and suppose the geometric automorphism group of $X$ is finite.
Then
\[
\sum_{C \in \Tw(X)} \frac{1}{\#\Aut C} = 1.
\]
\end{lemma}

\begin{proof}
For curves, this result appears as~\cite{vdgeer-vdvlugt}*{Prop.~5.1} 
and~\cite{katz-sarnak}*{Lem.~10.7.5}; for elliptic curves,
it is~\cite{howe:compositio1993}*{Prop.~2.1}.  We merely
sketch the proof here.
Let $G$ be the geometric automorphism group of $X$ and let 
$\varphi$ be the Frobenius automorphism of $\Kb$ over $K$.
We can define a (right) action of $G$ on itself by letting
an automorphism $\alpha$ act on $G$ by $a\mapsto\alpha^{-1}a\alpha^\varphi$.
The orbits of this action correspond to the elements of 
$H^1(\Gal \Kb/K,G)$, and hence to the twists of $X$.
Furthermore, if a twist $C$ corresponds to the cocycle that sends $\varphi$
to $a\in G$, then the automorphism group of $C$ is isomorphic to the stabilizer
of $a$ under the action of $G$ on itself that we just defined.
For each orbit $O$ of this action, let us choose a representative
element $a_O\in O$.  Clearly for each $O$ we have $(\#O)( \#\Stab a_O) = \#G$,
and we also clearly have
\[
\sum_{\text{orbits $O$}} \#O = \#G.
\]
Dividing this last equality by $\#G$ gives
\[
\sum_{C \in \Tw(X)} \frac{1}{\#\Aut C} = 
\sum_{\text{orbits $O$}} \frac{1}{\#\Stab a_O}
= 1.
\]
\end{proof}

\begin{remark}
The absolute Galois group of a finite field is isomorphic to~$\hat{\Z}$,
the profinite completion of the integers; this is the critical fact that
makes the discussion of twists of curves over finite fields simpler than
for curves over arbitrary fields.  We cannot resist observing that the
field $K=\BC((T))$  of Laurent series over the complex numbers 
also has $\hat{\Z}$ for its Galois group --- this follows from
Puiseux's theorem, which identifies $\bar{K}$ with 
$\cup_{n\ge1}\BC((T^{1/n}))$.  Thus, many of the arguments and examples
that we present below could easily be modified to work over $\BC((T))$ 
as well.
\end{remark}

\section{Two proofs of Theorem~\ref{T:main} for finite fields}
\label{S:proofs}
In this section we provide two proofs of Theorem~\ref{T:main} in the
case where $K_0$ is a finite prime field~$\F_p$.
Both proofs are based on the same basic strategy:
Given a prime $p$ and two coprime integers $r>1$ and $s>1$, 
we will find a power $q$ of $p$ and a curve $X$ over $K = \F_q$ such that
\begin{itemize}
\item the absolute Galois group $\Gal \Kb/K$ of $K$ acts trivially
      on the geometric automorphism group $G$ of $X$, and
\item the group $G$ contains two elements $x$ and $y$ such that
      $x^r$ and $y^r$ are conjugate to one another, and $x^s$ and $y^s$
      are conjugate to one another, but if $t$ is a proper divisor of
      $r$ or of $s$, then $x^t$ is not conjugate to $y^t$.
\end{itemize}
(Theorem~\ref{T:2rs}, below, shows that the order of such a group $G$
must be divisible by $rs$ and greater than~$rs$, and if $rs\equiv 2\bmod 4$
then the order of $G$ must be at least~$4rs$.)
Then we take $C$ and $D$ to be the $\Kb$-twists of $X$ corresponding
to $x$ and $y$, respectively.  Lemma~\ref{L:cohomology} 
shows that the degree-$r$ extension of $K$ and the degree-$s$ extension
of $K$ are both minimal isomorphism extensions for $C$ and~$D$.
The two proofs differ from one another in the choice of the
curve~$X$.

We choose to look for curves $X$ where $\Gal \Kb/K$ acts trivially
on $G$ purely for convenience; in this case the condition that
two cocycles give rise to twists of $X$ that have minimal isomorphism
extensions of degrees $r$ and $s$ turns into the easily-stated
conjugacy condition given above.  In Section~\ref{S:genus2} we will
analyze a number of examples in which the Galois group does
not act trivially.

\begin{proof}[First proof of Theorem~\textup{\ref{T:main}} for finite fields]
Let $K_0=\F_p$, $r$, and $s$ be given, and interchange $r$ and $s$, if
necessary, so that $r$ is odd. Let $D_{4rs}$ denote the dihedral 
group of order $4rs$, and let $u$ and $v$ be elements of $D_{4rs}$
satisfying $u^{2rs} = v^2 = 1$ and $vuv = u^{-1}$.  Note that $u^i$ is
conjugate to $u^j$ in $D_{4rs}$ if and only if 
$j \equiv \pm i \bmod 2rs$.

Let $m$ be an integer that is congruent to $1$ modulo $r$ and congruent 
to $-1$ modulo $2s$,  let $y = u^m$, and let $x=u$.  Suppose $i$ is an integer such
that $x^i$ is conjugate to $y^i$.  Then either
$i \equiv im \bmod 2rs$ or $i\equiv -im \bmod 2rs$.
The first possibility occurs precisely when $s$ divides $i$, and the second
when $r$ divides $i$.

A result of Madden and Valentini~\cite{madden-valentini} says that every 
finite group occurs as the automorphism group of a curve over~$\Fbar_p$.
Thus, there is a curve $X$ over $\Fbar_p$ whose automorphism group is 
isomorphic to $D_{4rs}$.  The curve $X$ may be defined over some finite field
$K\subset \Fbar_p$, and by replacing $K$ by a finite extension, we may assume that 
all of the geometric automorphisms of $X$ are defined over~$K$.

Proceeding as in the outline presented earlier, we find that there are
twists $C$ and $D$ of $X$ that have minimal isomorphism extensions
of degrees $r$ and $s$ over~$K$.
\end{proof}

This first proof is straightforward, but says nothing about the genus or
field of definition of the examples.  The next proof shows that we can 
find examples over~$\F_{p^2}$.

\begin{proof}[Second proof of Theorem~\textup{\ref{T:main}} for finite fields]
Let $K_0=\F_p$, $r$, and $s$ be given, and interchange $r$ and $s$, if necessary,
so that $r$ is odd.  Let $n$ be a positive integer that is coprime to $p$,
that is divisible by at least two odd primes, and that has at least one
prime divisor that is congruent to $1$ modulo~$2rs$.
Let $q$ be an even power of $p$, let
$t$ be the positive square root of~$q$, and fix a primitive $n$-th root
of unity $\zeta\in\Fbar_q$.  Let $X$ be the modular curve over 
$K = \F_q$ with the following property: for every finite extension 
$M = \F_{q^e}$ of~$K$, the noncuspidal $M$-rational points of $X$ 
parametrize triples $(E,P,Q)$, where $E$ is an elliptic curve over $M$
such that the $\F_{q^e}$-Frobenius acts on $E[n]$ as multiplication 
by~$t^e$, and where $P$ and $Q$ are $\Fbar_q$-rational points of order 
$n$ on $E$ such that the Weil pairing of $P$ and $Q$ is~$\zeta$.  (Thus
$X$ is an $\F_q$-rational version of the usual modular curve~$X(n)$,
constructed in much the same way as the Igusa-Ihara modular curve is
constructed --- see~\cite{ihara}.)  Let 
$G = \Aut X_{\Fbar_q}$ and let $G_0 = \Aut X$.  From the modular interpretation
of $X$ it is easy to see that $G_0$ contains a group isomorphic to 
$\SL_2(\Z/n\Z) / \langle\pm1\rangle$.  But the main result 
of~\cite{mod2} implies that $G\cong \SL_2(\Z/n\Z) / \langle\pm1\rangle$,
so $G$ is equal to~$G_0$.  Therefore the Galois group of $\Kb/K$ acts
trivially on $G$.

Let $\mu$ be an element of $(\Z/n\Z)^*$ of order $2rs$ such that
$\mu^{rs} \neq -1.$  (The conditions on $n$ imply that such a $\mu$ exists.)
Let $u$ be the matrix 
  $(\!\begin{smallmatrix}\mu &0 \\0 &1/\mu\end{smallmatrix}\!)$ 
in $\SL_2(\Z/n\Z)$, and let $v$ be the matrix
  $(\!\begin{smallmatrix}0 &1 \\-1 &0\end{smallmatrix}\!)$.
It is easy to see that the subgroup of 
  $G = \SL_2(\Z/n\Z) / \langle\pm1\rangle$
generated by the images of $u$ and $v$ is isomorphic to the dihedral
group $D_{4rs}$.

Let $m$ be an integer that is congruent to $1$ modulo $r$ and congruent 
to $-1$ modulo $2s$, and let $x$ and $y$ be the images
of $u$ and $u^m$ in $G$.  A straightforward computation shows 
that if $d$ is a proper divisor of $r$ or of $s$, 
then $u^d$ and $u^{md}$ have distinct sets of eigenvalues,
even up to sign, so $x^d$ and $y^d$ are not conjugate to one another in~$G$.
But $x^r$ and $y^r$ (and $x^s$ and $y^s$) are conjugate to one
another in $D_{4rs}$, whence also in~$G$.

As before, if we let $C$ and $D$ be the $L/K$-twists of $X$ given by the
elements $x$ and $y$ of $G$, then $C$ and $D$ satisfy the conclusion of 
Theorem~\ref{T:main}.
\end{proof}

The preceding proof gives examples over the quadratic extension of $K_0$, but with
genus greater than $(2rs)^3$.  In the following section, we will show
that when $p$ does not divide $rs$ we can construct explicit examples
(over possibly large extensions of $K_0$) of genus $rs-1$.

\section{Explicit examples, and the proof of Theorem~\ref{T:main} over $\Q$}
\label{S:proof3}

In this section we provide explicit examples of curves satisfying the
conclusion of Theorem~\ref{T:main} for certain values $r$ and $s$
and for certain prime fields $K_0$.
In particular, Constructions~\ref{Con:prsodd} and~\ref{Con:prseven}
give explicit curves for the case in which the characteristic $p$
of $K_0$ does not divide $rs$ (and therefore provide a proof of 
Theorem~\ref{T:main} for the case $K_0=\Q$), and Construction~\ref{Con:ps}
gives a construction for the case in which $r = p>0$ and $s$ is prime.

\begin{construction}
\label{Con:prsodd}
Let $K_0$ be a prime field whose characteristic $p$ is not~$2$.
Let $r>1$ and $s>1$ be integers that are coprime to one another, 
that are not divisible by $p$, and with $r$ odd.
Let $K$ be a finite extension of $K_0$ that contains the
$4rs$-th roots of unity and let $a\in K^*$
be an element whose image in $K^*/K^{*2rs}$ has order $2rs$.
Using the Chinese Remainder Theorem, choose integers $i$ and $j$ as
follows\/\textup{:}  If $s$ is odd, let $i$ and $j$ satisfy
\begin{align*}
i &\equiv \phantom{s+1}\llap{$1$}  \bmod r  &  j &\equiv \phantom{s+1}\llap{$-1$} \bmod r  \\
i &\equiv                   s+1    \bmod 2s &  j &\equiv                    s+1   \bmod 2s,
\intertext{while if $s$ is even, let $i$ and $j$ satisfy}
i &\equiv 1\bmod r    &  j &\equiv                   -1  \bmod r  \\
i &\equiv 1\bmod 2s   &  j &\equiv \phantom{-1}\llap{$1$}\bmod 2s.
\end{align*}
Let $C$ and $D$ be the curves over $K$ defined by
\begin{alignat*}{2}
C:&\quad  & z^2 & = w^{2rs} + a^i\\
D:&\quad  & v^2 & = u^{2rs} + a^j.
\end{alignat*}
Then $C$ and $D$ are twists of one another, and they have
minimal isomorphism extensions of degrees $r$ and $s$ over $K$.
\end{construction}

\begin{proof}
Note that whatever the parity of $s$, we always have
\[
(j-i,2rs) = 2s\text{\qquad and\qquad} (j+i,2rs) = 2r,
\]
and $sj$ is always even.  Also, $-1$ is a $2rs$-th power in $K^*$.
Let $E$ be the Kummer extension $K(a^{1/2rs})$ of $K$, let $e$ be
an element of $E$ with $e^{2rs} = a$, let $L = K(e^{2s})$, and let
$M = K(e^{2r})$, so that $L$ and $M$ are extensions of $K$ of
degrees $r$ and $s$, respectively.
We start by showing that $C$ and $D$ become
isomorphic to one another over $L$ and over~$M$.

Let $e_L = e^{2s}\in L$, so that $e_L^r = a$. 
Set $c = e_L^{(j-i)/(2s)}$ and $d = c^{rs}$. Then $u= cw$, $v = d z$ 
gives an isomorphism from $C_L$ to $D_L$.

Let $e_M = e^{2r}\in M$, so that $e_M^s = a$.  Set $c = e_M^{(j+i)/(2r)}$
and $d = e_M^{sj/2}$.  Then $u= c/w$, $v = d z / w^{rs}$ gives an isomorphism
from $C_M$ to $D_M$.

Now let $N$ be a finite extension of $K$ that is a proper subextension of
either $L$ or~$M$;  we must show that $C_N$ and $D_N$ are not
isomorphic to one another. 
Suppose, to obtain a contradiction, that
there is an isomorphism $\varphi$ from $C_N$ to $D_N$.  Then $\varphi$ induces an 
isomorphism $\bar{\varphi}$ from $\PP^1_N$ to $\PP^1_N$ that takes the roots
of the polynomial  $f = x^{2rs} + a^i\in N[x]$ to the roots of
$g = x^{2rs} + a^j$.  Kummer theory shows that if $d$ is the degree 
of $N$ over $K$, then the polynomial $f$
splits into irreducible factors of degree $rs/d$ (if $i$ is even)
or $2rs/d$ (if $i$ is odd).  In either case, these irreducible factors
have degree at least $6$.
Lemma~\ref{L:kummer} below shows that $\bar{\varphi}$ must be of the
form $u = cw$ or $u = c/w$ for some $c\in N^*$.
Thus our hypothetical map $\varphi$ from $C$ to $D$ is either of the form
\[
u = cw \qquad v = dz \qquad\text{for some $c,d \in N$}
\]
or of the form
\[
u = c/w \qquad v = dz/w^{rs} \qquad\text{for some $c,d \in N$.}
\]

If $\varphi$ has the former shape, we find that we must have
both $d^2 = c^{2rs}$ and $d^2 a^i  = a^j$, so that 
$a^{j-i} = c^{2rs}$.   Now, $j-i = 2st$ for some $t$ coprime to $r$,
so we find that $a^t = \zeta c^r$ for some $\zeta\in K$ with $\zeta^{2s} = 1$.
Since $K$ contains the $2rs$-th roots of unity, there is a $\xi\in K$ with
$\zeta = \xi^r$.  Thus, the element $\xi c$ of $N$ satisfies
$(\xi c)^r = a^t$.  Now, since the image of $a$ in $K^*/K^{*2rs}$ has order $2rs$,
the image of $a$ in $K^*/K^{*r}$ has order $r$, as does the image of $a^t$.
By Kummer theory, the degree of $K(\xi c)$ over $K$ is $r$, so the degree of 
$N$ over $K$ is divisible by $r$, a contradiction.

If $\varphi$ is of the form $u=c/w$, $v=dz/w^{rs}$, then 
we have both $d^2 = a^j$ and $d^2 a^i = c^{2rs}$, so that
$a^{j+i} = c^{2rs}$.  Now, $i+j\equiv 0\bmod r$ and $i+j\equiv 2 \bmod 2s$,
so $i+j = 2st$ for some $t$ coprime to $s$.  Arguing as in the preceding paragraph, 
we find that the degree of $N$ over $K$ is divisible by $s$, a contradiction.
\end{proof}

\begin{lemma}
\label{L:kummer}
Let $m>2$ be an integer, let $K$ be a field whose characteristic does not
divide $m$, and suppose that $K$ contains the $m$-th roots of unity.
Suppose $a\in K^*$ is not an $m$-th power and let $b$ be a nonzero element of $K$.
Then any $K$-rational automorphism $\psi:\PP^1_K\to\PP^1_K$ that takes the set of
roots in $\Kb$ of the polynomial $x^m-a\in K[x]$ to the set of roots of $x^m - b$ is 
of the form $x\mapsto cx$ or $x\mapsto c/x$ for some nonzero $c\in K$.
\end{lemma}

\begin{proof}
Let $L$ be the splitting field of $x^m - a$ over $K$, let $\sigma$ be a 
generator for $\Gal L/K$, and let $\zeta$ be the $m$-th root of unity
such that $\alpha^\sigma = \zeta \alpha$ for all roots $\alpha$ of $x^m-a$.
Let $d$ be the multiplicative order of $\zeta$, so that $d$ is the degree
of the irreducible factors of $x^m-a$ in $K[x]$.
Note that $d>1$ because $a$ is not an $m$-th power.
Since $\psi$ is $K$-rational, there is a primitive $d$-th root of unity $\xi$
in $K$ such that $\beta^\sigma = \xi \beta$ for all roots $\beta$ of $x^m-b$.

Let $\alpha$ and $\beta$ be roots of $x^m-a$ and $x^m-b$, respectively,
such that $\psi(\alpha) = \beta$.
The fact that $\psi$ is $K$-rational implies that 
$\psi(\alpha^\tau) = \beta^\tau$ for all $\tau\in\Gal L/K$,
so we have $\psi(\zeta^i\alpha) = \xi^i \beta$ for all integers~$i$.
Let $\chi$ be the automorphism of $\PP^1_\Kbar$ 
such that $\chi(x) = \psi(\alpha x) / \beta$.  Then $\chi(\zeta^i) = \xi^i$ for
all $i$.
Let $r,s,t,u$ be elements of $\Kb$ such that
$\chi(x) = (rx+s)/(tx+u).$  
The conditions that  $\chi(\zeta^i) = \xi^i$ for $i\in\{0,1,2,3\}$
show that $[r,s,-t,-u]$ is an element of the null space of the Vandermonde matrix
\[\left[\begin{matrix}
1       & 1 & 1            & 1     \\
\zeta   & 1 & \zeta  \xi   & \xi   \\
\zeta^2 & 1 & \zeta^2\xi^2 & \xi^2 \\
\zeta^3 & 1 & \zeta^3\xi^3 & \xi^3 
\end{matrix}\right].\]
Therefore the determinant of this matrix is $0$, and it follows that
either $\zeta = \xi$ or $\zeta\xi = 1$.

Suppose that $d>2$.  If $\zeta=\xi$ then $\chi$ agrees
with the automorphism $x\mapsto x$ for three distinct values
of $x$ (namely $1$, $\zeta$, and $\zeta^2$), so 
$\chi(x) = x$ and $\psi(x) = c x$ with $c= \beta/\alpha$.
If $\zeta \xi = 1$ then $\chi$ agrees with the automorphism $x\mapsto 1/x$ 
for three distinct values of $x$, so $\chi(x) = 1/x$ and $\psi(x) = c/x$
with $c= \beta\alpha$.

Suppose that $d=2$.  Then $\psi(-\alpha) = -\psi(\alpha)$ for
every root $\alpha$ of $x^m-a$, and it follows that $\chi(-\eta) = -\chi(\eta)$
for every $m$-th root of unity $\eta$.   Since $\chi(x) = (rx+s)/(tx+u)$,
we find that $rt\eta^2 = su$ for every $m$-th root of unity $\eta$.  
Since $m>2$, we must have $rt=su=0$, and we see that either $\psi(x) = cx$ or
$\psi(x) = c/x$ for some constant $c$.
\end{proof}

\begin{remark}
The lemma would be false without the assumption that 
$a$ is not an $m$-th power.  For example, suppose $r$ and $s$
are nonzero elements of $\F_p$ with $r^2 \neq s^2$.  
Then the automorphism $x\mapsto (rx + s)/(sx + r)$
permutes the roots of $x^{p+1} - 1$.  However, the lemma would remain
true if we replaced the hypothesis about $a$ with the hypothesis that
$m\not\equiv 0,1 \bmod p$. 
\end{remark}

\begin{construction}
\label{Con:prseven}
Let $r>1$ and $s>1$ be odd integers that are coprime to one another.
Let $q$ be a power of $4$ that is congruent to $1$ modulo~$rs$ and
let $a$ be a generator of\/ $\F_q^*$.  Let $m$ be an integer that is
congruent to $-1$ modulo $r$ and to $1$ modulo $s$.
Let $C$ and $D$ be the curves over $K=\F_q$ defined by
\begin{alignat*}{2}
C:&\quad  & z^2 + z & = \frac{a}{w^{rs} + a}\\
D:&\quad  & v^2 + v & = \frac{a^m}{u^{rs} + a^m}.
\end{alignat*}
Then $C$ and $D$ are twists of one another, and they have
minimal isomorphism extensions of degrees $r$ and $s$ over $K$.
\end{construction}

\begin{proof}
We start by showing that $C$ and $D$ become
isomorphic to one another over the extensions of $K$ of degrees $r$ and $s$.

Let $L$ be the degree-$r$ extension of $K$ and let $e$ be an element of $L$
with $e^r = a$.  Set $c = e^{(m - 1)/s}$. Then
$u= cw$, $v = z$  gives an isomorphism from $C_L$ to $D_L$.

Let $M$ be the degree-$s$ extension of $K$ and let $e$ be an element of $M$
with $e^s = a$.  Let $\omega$ be an element of $K$ such that $\omega^2 + \omega = 1$.
Set $c = e^{(m+1)/r}$.  
Then $u= c/w$, $v = z +\omega$ gives an isomorphism from $C_M$ to $D_M$.

Now let $N$ be a finite extension of $K$ of whose degree $d$ is a proper
divisor of $r$ or of~$s$; we must show that $C_N$ and $D_N$ are not
isomorphic to one another.  Suppose, to obtain a contradiction, that
there is an isomorphism $\varphi$ from $C_N$ to $D_N$.  Then $\varphi$ induces an 
isomorphism $\bar{\varphi}$ from $\PP^1_N$ to $\PP^1_N$ that takes the roots
of the polynomial  $f = x^{rs} + a\in N[x]$ to the roots of
$g = x^{rs} + a^m$, and since $f$ has no roots in $N$ we see from 
Lemma~\ref{L:kummer} that $\bar{\varphi}$ is either of the form $x\mapsto cx$ or
$x\mapsto c/x$ for some constant $c\in N$.

If $\bar{\varphi}(x) = cx$, then the roots of $(cx)^{rs} + a^m$ must be the
roots of $x^{rs} + a$, so $a^{m-1} = c^{rs}$.  As in the proof of 
Construction~\ref{Con:prsodd}, we find that $a$ must be an $r$-th power
in $N$, so $d$ is a multiple of $r$, a contradiction.
Similarly, if $\bar{\varphi}(x) = c/x$ then we find that $a^{m+1} = c^{rs}$,
so that $a$ is an $s$-th power in $N$, and $d$ is a multiple of $s$,
again a contradiction.
\end{proof}

\begin{remark}
The curves in Constructions~\ref{Con:prsodd} and~\ref{Con:prseven}
have genus $rs-1$.
\end{remark}

\begin{construction}
\label{Con:ps}
Let $p$ and $s$ be distinct prime numbers, let $q$ be a power of $p$
that is congruent to $1$ modulo~$s$, and let $a$ be a generator of\/ $\F_q^*$.
Let $C$ and $D$ be the curves over $K=\F_q$ defined by
\begin{alignat*}{2}
C:&\quad  & z^q - z & =   w^s - 1  \\
D:&\quad  & v^q - v & = a u^s - 1.
\end{alignat*}
Then $C$ and $D$ are twists of one another, and they have
minimal isomorphism extensions of degrees $p$ and $s$ over $K$.
\end{construction}

\begin{proof}
Let $L$ be the degree-$p$ extension of $K$ and let $e$ be an element of $L$
with $e^q - e = a - 1$.  Then
$u= w$, $v = az + e$  gives an isomorphism from $C_L$ to $D_L$.

Let $M$ be the degree-$s$ extension of $K$ and let $e$ be an element of $M$
with $e^s = a$.  Then $u= w/e$, $v = z$ gives an isomorphism from $C_M$ to $D_M$.

To complete the proof, we must show that $C$ and $D$ are not isomorphic to one
another over $K$.  To see this, we can simply note that $C$ and $D$ have different
numbers of $K$-rational points.  Both curves have a single rational point 
lying over the (singular) point at infinity in the models given above.
The curve $D$ has no further rational points, because if $u$ and $v$ are elements
of $K$ then $v^q - v = 0$ but $a u^s \neq 1$, because $a$ is not an $s$-th power.
On the other hand, $C$ has $sq$ further rational points: $z$ can be an arbitrary
element of $K$, and $w$ can be an arbitrary $s$-th root of unity.
\end{proof}

\begin{remark}
The curves in this construction have genus $(q-1)(s-1)/2$.
\end{remark}

One can relate these constructions to our discussion of automorphism groups.
For example, the curves $C$ and $D$ from Construction~\ref{Con:prsodd}
are both twists of the curve $X$ over $\F_q$ defined by 
$y^2 = x^{2rs} + 1$.  Let $a$ be the generator of $\F_q^*$ chosen
in Construction~\ref{Con:prsodd}, let $\alpha\in\Fbar_q$ satisfy
$\alpha^{2rs} = a$, and let $\zeta$ be the primitive $2rs$-th
root of unity $\alpha^{q-1}$.
The curve $X$ has some obvious automorphisms $\rho,\varphi,\eta$ defined by
\begin{align*}
\rho(x,y)    & = (\zeta x,y)     \\
\varphi(x,y) & = (1/x, y/x^{rs}) \\
\eta(x,y)    & = (x, -y).
\end{align*}
The subgroup $G$ of $\Aut X$ generated by these automorphisms has 
order~$8rs$, and one can show that when $2rs \not\equiv1\bmod p$
the curve $X$ has no geometric automorphisms other than these.  The 
curve $C$ is the twist of $X$ by $(\eta\rho)^i$ and the curve $D$ is the 
twist of $X$ by $(\eta\rho)^j$.  One can compute that the $r$-th powers
of $(\eta\rho)^i$ and $(\eta\rho)^j$ are conjugate to one another in~$G$,
as are their $s$-th powers, but their $d$-th powers are not conjugate to one 
another when $d$ is a proper divisor of $r$ or of~$s$.

Similar computations can be made for the curves that appear
in Constructions~\ref{Con:prseven} and~\ref{Con:ps}.

\begin{remark}
The constructions in this section
depend rather visibly on the existence of elements of
order $r$ and $s$ in the group $\Aut \PP^1_{\Kbar_0}$. 
When $K_0 = \F_p$ this group contains no elements of order $p^2$,
so any explicit construction that deals with general values of $r$ 
and $s$ will have to use ideas not present in this section.
\end{remark}

\begin{remark}
Suppose $r$ and $s$ are distinct prime numbers, and let $p$ be either 
$0$ or a prime.  Then we can apply one of the constructions given in 
this section to produce 
explicit equations for curves $C$ and $D$ over a field of characteristic~$p$
that satisfy the conclusion of Theorem~\ref{T:main}.  Thus, this section 
provides a proof of Theorem~\ref{T:main} in the special case where $r$ and
$s$ are prime.
\end{remark}

\section{The nonexistence of genus-$0$ examples}
\label{S:genus0}
In this section we give a proof of Theorem~\ref{T:genus0} that uses no Galois 
cohomology.  Before we start, we note that there is a short proof based on
quaternion algebras and the Brauer group.
First, one can relate twists of the projective line
to quaternion algebras (as in~\cite{serre:GC}*{{\S}III.1.4});
the proof then reduces to the problem of showing that two quaternion
algebras $H_1$ and $H_2$ over $K$ that become isomorphic to one another
over an odd-degree extension $L$ of $K$ are already isomorphic over $K$.  
The classes of $H_1$ and $H_2$ in the Brauer group $\Br(K)$ of $K$ have order~$2$,
so we would like to show that the $2$-torsion element $[H_1]-[H_2]$ of~$\Br(K)$,
which becomes trivial in $\Br(L)$, is trivial already in~$\Br(K)$. 
If $L$ is a separable extension of $K$ there is an easy argument 
that shows this; if $L$ is inseparable over $K$,
we can use~\cite{serre:LF}*{Exer.~X.4.2}.  For those
readers familiar with the concepts involved, this is a reasonably 
direct method of proof.
The proof we present here is slightly longer, but it uses much less
machinery.

We start with some basic facts about curves of genus~$0$.

\begin{lemma}
\label{L:conic}
Let $C$ be a genus-$0$ curve over a field $K$.  Then there is an
embedding of $C$ into\/ $\PP^2$ as a nonsingular conic.  Also, $C$ is 
isomorphic to\/ $\PP^1$ if and only if it has a rational point.
\end{lemma}

\begin{proof}
The canonical divisor class of $C$ has degree $-2$.  Let $D$ be $-1$ 
times a canonical divisor.  Then the Riemann-Roch formula shows that for
every positive integer $n$, we have $\ell(n D) = 2n + 1$. In particular,
we find that there are three linearly independent functions $x$, $y$, 
$z$ in $L(D)$; furthermore, since $\ell(2D) = 5$, there must be a 
relation among the $6$ functions $x^2, xy, xz, y^2, yz, z^2$.  This 
relation defines a conic $C'$ in $\PP^2$ and a map $C\to C'$.  If $C'$ 
were singular, its defining equation would factor into the product of 
two linear terms, contradicting the fact that $x$, $y$, and $z$ are
linearly independent.  Using the fact that $\ell(n D) = 2n + 1$ it is 
not hard to show that the functions $x^i y^j z^k$ with $i + j + k = n$ 
span $L(nD)$, and it follows that the function field of $C$ is generated
by $x$, $y$, and $z$.  Therefore, the map $C \to C'$ is an isomorphism.

If $C$ is isomorphic to $\PP^1$ then it has a rational point. 
Conversely, if $C\cong C'$ has a rational point, then projecting $C'$
away from this rational point onto our favorite copy of $\PP^1$ in 
$\PP^2$ will give us an isomorphism from $C$ to $\PP^1$.
\end{proof}

Next, we require a result about quadratic forms.

\begin{lemma}[Springer]
\label{L:Springer}
Let $Q$ be a quadratic form over a field $K$.  If $Q$ represents $0$
over an odd-degree extension $L$ of $K$, then it represents $0$ 
over~$K$.
\end{lemma}

\begin{proof}
For fields of characteristic not $2$, this is proven in~\cite{springer}
(and reproduced, for example, in~\cite{scharlau}*{Thm.~2.5.3}).  But the
proof is surprisingly simple, and works in characteristic $2$ as well, 
so we present the proof here for the reader's convenience.

We argue by contradiction.  Let $n$ be the smallest odd integer for
which there is a field $K$, a degree-$n$ extension $L$ of $K$, and a
quadratic form $Q$ over $K$ for which the lemma fails.  (Note that 
clearly $n>1$.)  The minimality of $n$ shows that there are no fields 
intermediate between $L$ and $K$, so there is a primitive element 
$\alpha$ for $L$ over $K$.  Let $p\in K[t]$ be the minimal polynomial 
for $\alpha$.

Let $m$ be the number of variables for the quadratic form~$Q$, and 
suppose $\beta_1,\ldots,\beta_m$ are elements of $L$, not all zero, such
that $Q(\beta_1,\ldots,\beta_m) = 0.$  Let $f_1,\ldots,f_m\in K[t]$ be 
polynomials of degree at most $n-1$ such that $\beta_i = f_i(\alpha)$ 
for all~$i$.  Then in the polynomial ring $K[t]$ we have 
\[
Q(f_1,\ldots,f_m) \equiv 0\bmod p,
\]
but $f_i\not\equiv 0 \bmod p$ for some $i$.  These relations will still
hold if we divide each $f_i$ by the greatest common divisor of all of 
the $f_i$, so we may assume that the $f_i$ have no nontrivial common
factor.

Let $k$ be the largest degree of the $f_i$, and for each $i$ let $b_i$ 
be the coefficient of $t^k$ in $f_i$.  Then 
\[
Q(f_1,\ldots,f_m) = Q(b_1,\ldots,b_m) t^{2k} + 
                      \text{(lower order terms)},
\]
and since $Q$ does not represent $0$ over $K$ we see that 
$Q(f_1,\ldots,f_m)$ is a polynomial of degree~$2k<2n$.  Thus we have 
$Q(f_1,\ldots,f_m) = p q$ for some polynomial $q$ whose degree is odd 
and is less than $n$.  One of the irreducible factors of $q$ must also 
have odd degree less than $n$.  Let $r$ be one such factor of~$q$.

Let $L'$ be the extension of $K$ defined by~$r$, and let $\alpha'$ be a 
root of $r$ in~$L'$.  For each $i$ let $\beta_i' = f_i(\alpha')$.  The 
$\beta_i'$ are not all $0$ because the $f_i$ are not all divisible 
by~$r$.  But $Q(\beta_1',\ldots,\beta_m') = 0$, so $L'$ is an extension 
of $K$ of odd degree less than $n$ for which the lemma fails.  This 
contradicts the definition of~$n$.
\end{proof}

\begin{remark}
For an amusing exercise, the reader should determine how the proof fails
without the assumption that $n$ is odd.
\end{remark}

As a corollary to Springer's theorem, we get a special case of 
Theorem~\ref{T:genus0}.

\begin{corollary}
\label{C:conic}
Let $C$ be a genus-$0$ curve defined over a field $K$.  If $C$ has a 
point over an odd-degree extension $L$ of $K$, then $C$ has a point 
over~$K$.
\end{corollary}

\begin{proof}
Lemma~\ref{L:conic} shows that $C$ has can be written as a conic 
in~$\PP^2$, so there is a $3$-variable quadratic form $Q$ that gives an
equation for $C$.  If $C$ has a point over $L$ then this quadratic form
has a nontrivial zero over $L$.  Springer's theorem then shows that the
quadratic form has a nontrivial zero over $K$, so $C$ has a point 
over~$K$.
\end{proof}

We need one more lemma before we can prove Theorem~\ref{T:genus0}.

\begin{lemma}
\label{L:conicautomorphism}
Suppose $C$ is a genus-$0$ curve over a field $K$, and let $M$ be a
separable quadratic extension of $K$ over which $C$ has points.  Let $P$
and $Q$ be points of $C(M)$ that do not lie in $C(K)$ and that 
are not\/ $\Gal M/K$-conjugates of one another.  Then there is a 
\textup{(}$K$-defined\/\textup{)} automorphism $\alpha$ of $C$ such that
$\alpha(P) = Q$.
\end{lemma}

\begin{proof}
If $P = Q$ we simply take $\alpha$ to be the identity, so we may assume
that $P\neq Q$.

Let $\sigma$ be the nontrivial element of $\Gal M/K$.  Consider the 
subvariety $V$ of the variety $\Aut(C)$ defined by
\[
\{ \alpha \in \Aut(C) : \alpha(P) = Q, 
                          \quad\alpha(P^\sigma) = Q^\sigma,
                          \quad\alpha^2 = 1\}.
\]                          
The conditions defining $V$ are stable under $\Gal M/K$, so $V$ is
defined over~$K$.  Over $M$, there is a constant $a\neq 0,1$ such that
$V$ is isomorphic to the subvariety of $\Aut \PP^1$ consisting of 
involutions that send $0$ to $\infty$ and $1$ to $a$; clearly, this
latter variety has exactly one geometric point, corresponding to the
involution $x \mapsto a/x$.  So $V$ has one geometric point, which is
defined over~$M$.  It follows that this single point is defined 
over~$K$.  The lemma follows.
\end{proof}

\begin{proof}[Proof of Theorem~\textup{\ref{T:genus0}}]
We are given genus-$0$ curves $C$ and $D$ over a field $K$ that become
isomorphic to one another over an odd-degree extension $L$ of $K$, and
we are to show that $C$ and $D$ are isomorphic over $K$.  If either $C$
or $D$ is isomorphic to $\PP^1$, then the desired result follows from 
Corollary~\ref{C:conic} and Lemma~\ref{L:conic}.  Thus we may assume
that neither $C$ nor $D$ has a $K$-rational point.

Let $M$ be a separable quadratic extension over which $C$ has rational
points; the fact that $C$ can be represented as a nonsingular conic
shows that such fields exist.  Let $P$ be an element of $C(M)$, and let
$N$ be the compositum of $L$ and $M$ over $K$.  Let $\varphi$ be an
isomorphism from $C_L$ to $D_L$.  Then $\varphi(P)$ is an $N$-rational
point of $D$.  

We see that the curve $D_M$ over $M$ has an $N$-rational point, where
$\indexin{N}{M} = \indexin{L}{K}$ is odd.  Corollary~\ref{C:conic} shows
that therefore $D_M$ has a point over $M$; in other words, there is an
$M$-rational point $Q$ on $D$.  Replacing $Q$ with its $M/K$-conjugate,
if necessary, we may assume that $Q$ is not the $N/L$-conjugate 
of~$\varphi(P)$.

Applying Lemma~\ref{L:conicautomorphism} to the curve $D_L$ over $L$ and
the $N$-rational points $\varphi(P)$ and $Q$ of $D_L$, we find that 
there is an automorphism $\alpha$ of $D_L$ that sends $\varphi(P)$ 
to~$Q$.  Replacing $\varphi$ with $\alpha\varphi$, we find that 
$\varphi$ is an isomorphism $C_L \to D_L$ that takes $P$ to~$Q$.

Consider the subvariety $V$ of the variety $\Hom(C,D)$ defined by
\[
\{ \psi \in \Hom(C,D) : \psi(P) = Q, 
                          \quad\psi(P^\sigma) = Q^\sigma\},
\]                  
where again we take $\sigma$ to be the nontrivial element of $\Gal M/K$.
The conditions defining $V$ are stable under $\Gal M/K$, so $V$ is 
defined over~$K$.  Over $M$, there is a constant $a\neq 0,1$ such that
$V$ is isomorphic to the subvariety of $\Aut \PP^1$ consisting of 
automorphisms that send $0$ to $\infty$ and $1$ to $a$; it is easy to 
see that this latter variety is isomorphic to $\PP^1$ minus two points.
Thus, our variety $V$ is isomorphic to a genus-$0$ curve $W$ with two
points removed.  Our isomorphism $\varphi:C_L\to D_L$ shows that the 
variety $V$ has a point over $L$, so the genus-$0$ curve $W$ has a point
over $L$ as well.  By Corollary~\ref{C:conic}, $W$ has points over~$K$,
and so is isomorphic to~$\PP^1$.  But over any field, $\PP^1$ minus two
points has rational points, so $V$ has a $K$-rational point.  Therefore 
$C$ is isomorphic to $D$ over~$K$.
\end{proof}

Over finite fields we can say more than is claimed by 
Theorem~\ref{T:genus0}. By Wedderburn's theorem, there are no nontrivial
quaternion algebras over a finite field.  Therefore, there are no 
nontrivial twists of $\PP^1$ over a finite field.

\section{Group theory}
\label{S:grouptheory}

In this section we state and prove several group-theoretical results
that are needed for the proof of Theorems~\ref{T:automorphism} 
and~\ref{T:genus1}.  First some notation: if $x$ and $y$ are elements of
a group~$G$, we write $x\sim y$ if $x$ and $y$ are conjugate in~$G$.

Suppose $G$ is a finite group and $\alpha$ is an automorphism of $G$.
We view $G$ as a topological group by endowing it with the discrete 
topology.  There is a continuous homomorphism $\hat{\Z}\to \Aut G$ that
sends $1$ to $\alpha$.  Using this action, we may consider the pointed
cohomology set $H^1(\hat{\Z},G)$ defined in Section~\ref{S:twists}; we
denote this pointed set by $H^1(\hat{\Z},G,\alpha)$ to specify the 
$\hat{\Z}$-action.  We identify a cocycle $a:\hat{\Z}\to G$ with 
$a(1)\in G$, and given a cocycle $x\in G$, we let $[x]_\alpha$ denote
its class in $H^1(\hat{\Z},G,\alpha)$.

\begin{theorem}
\label{T:cocycles}
Let $G$ and $\alpha$ be as above, and let $r>1$ and $s>1$ 
be integers that are coprime to one another.
Suppose $x$ and $y$ are two cocycles such that 
$[x x^\alpha \cdots x^{\alpha^{r-1}}]_{\alpha^r} = 
     [y y^\alpha \cdots y^{\alpha^{r-1}}]_{\alpha^r}$
and
$[x x^\alpha \cdots x^{\alpha^{s-1}}]_{\alpha^s} = 
     [y y^\alpha \cdots y^{\alpha^{s-1}}]_{\alpha^s}$
but such that for every proper divisor $d$ of $r$ or of $s$ we have
$[x x^\alpha \cdots x^{\alpha^{d-1}}]_{\alpha^d} \neq
     [y y^\alpha \cdots y^{\alpha^{d-1}}]_{\alpha^d}$.
Then the order of $G$ is divisible by $rs$ but is not equal to $rs$.  
Furthermore, if $rs\equiv 2\bmod 4$ then the order of $G$ is divisible by $2rs$.
\end{theorem}

Our proof of Theorem~\ref{T:cocycles} relies on three lemmas, which
we prove later in this section.  The first lemma allows us to rephrase
the statement of the theorem in terms of conjugacy in an extension of~$G$.

\begin{lemma}
\label{L:semidirect}
Suppose $G$ is a finite group and $\alpha$ is an element of\/ $\Aut G$.
Let\/ $\hat{\Z}$ act on $G$ by having $1\in\hat{\Z}$ act as $\alpha$.
Let $x$ and $y$ be elements of $G$ and let $m$ be a positive integer.
Then the cocycles\/ $\hat{\Z}\to G$
determined by $x x^\alpha \cdots x^{\alpha^{m-1}}$
and $y y^\alpha \cdots y^{\alpha^{m-1}}$ are cohomologous in 
$H^1(\hat{\Z},G,\alpha^m)$
if and only if the $m$-th powers of the
elements $(x,\alpha)$ and $(y,\alpha)$ of the 
semidirect product $A = G \rtimes \langle\alpha\rangle$ are conjugate to
one another in $A$.
\end{lemma}

The second lemma allows us to deduce much of Theorem~\ref{T:cocycles}
by looking at the conditions on $r$ and on $s$ separately.

\begin{lemma}
\label{L:conjugacy}
Let $A$ be a finite group with a normal subgroup $G$ such that
$A/G$ is cyclic.  Let $X$ and $Y$ be elements of $A$ that have the 
same image in $A/G$.  Suppose there is a positive integer $r$ 
such that $X^r\sim Y^r$ but $X^d\not\sim Y^d$
for all proper divisors $d$ of~$r$.  Then $r$ divides $\#G$.
\end{lemma}

Finally, the third lemma gives us stronger information at the prime $2$.

\begin{lemma}
\label{L:conjugacy2}
Let $A$ be a finite group with a normal subgroup $G$ such that
$A/G$ is cyclic.  Let $X$ and $Y$ be elements of $A$ that have the 
same image in in $A/G$.
Suppose that $X^2\sim Y^2$ and that $X^r\sim Y^r$ for some odd integer $r>1$,
but that $X\not\sim Y$.
Then $4$ divides $\#G$.
\end{lemma}

With these lemmas in hand, we can prove Theorem~\ref{T:cocycles}.

\begin{proof}[Proof of Theorem~\textup{\ref{T:cocycles}}]
Let $A$ be the semidirect product $G \rtimes \langle\alpha\rangle$
and let $X$ and $Y$ be the elements $(x,\alpha)$ and $(y,\alpha)$ of~$A$.
Lemma~\ref{L:semidirect} shows that
the hypothesis of the theorem is equivalent to the statement that 
$X^r\sim Y^r$  and $X^s\sim Y^s$ but that $X^d\not\sim Y^d$ for all 
proper divisors $d$ of $r$ or of~$s$.

Lemma~\ref{L:conjugacy} shows that $\#G$ is divisible by both
$r$ and $s$, and since $r$ and $s$ are coprime, we find that $rs$ divides~$\#G$.
Suppose one of $r$ and $s$ is congruent to $2$ modulo $4$, say $s\equiv 2\bmod 4$.
Applying Lemma~\ref{L:conjugacy2} to $X^{s/2}$ and $Y^{s/2}$, we see that
$\#G$ is divisible by $4$, so that $2rs$ divides~$\#G$.

Suppose that $\#G = rs$.  The element $X$ of $A$ acts on the normal subgroup
$G$ of $A$ by conjugation.  Theorem~$2$ of~\cite{horosevskii} says that an automorphism
of a nontrivial finite group has order less than that of the group
(this is also an easy consequence of~\cite{guralnick}*{Thm.~1}), so
there is a positive $m < rs$ 
such that $X^m$ acts trivially on $G$.  Since $r$ and $s$ are coprime to one another,
they both cannot divide $m$.  Switch $r$ and $s$, if necessary, so that $r$ does
not divide $m$.  Replacing $Y$ by a conjugate, we may assume that $X^r = Y^r$.

Since $X^m$ commutes with $X$, and since $A$ is generated by $G$ and $X$, we see that
$X^m$ lies in the center of $A$.  Since $Y^{mr}$ is conjugate to $X^{mr}$, we see that
$Y^{mr}=X^{mr}$.  Likewise, $Y^{ms}=X^{ms}$.  Since $r$ and $s$ are coprime to one
another, we find that $Y^m = X^m$.  Combining this with the fact that $X^r = Y^r$,
we find that $X^g = Y^g$, where $g = (m,r)$ is a proper divisor of $r$.  This
contradiction shows that~$\#G > rs$.
\end{proof}

We are left with the proofs of our three lemmas.

\begin{proof}[Proof of Lemma~\textup{\ref{L:semidirect}}]
We know from Section~\ref{S:twists} that the cocycles 
$x x^\alpha \cdots x^{\alpha^{m-1}}$
and $y y^\alpha \cdots y^{\alpha^{m-1}}$ 
are cohomologous if and only if there is a $z\in G$ such that 
\[
y y^\alpha \cdots y^{\alpha^{m-1}} = z^{-1}
                                     x x^\alpha \cdots x^{\alpha^{m-1}}
                                     z^{\alpha^m}.
\]                                     

Suppose the two cocycles are cohomologous, and let $z$ be as above.
Then in $A$ we have 
\begin{align*}
(y y^\alpha \cdots y^{\alpha^{m-1}} ,\alpha^m) 
  & = (z^{-1} x x^\alpha \cdots x^{\alpha^{m-1}} z^{\alpha^m}, \alpha^m)\\
  & = (z, 1)^{-1} (x x^\alpha \cdots x^{\alpha^{m-1}},\alpha^m) (z, 1).
\end{align*}
But $(y y^\alpha \cdots y^{\alpha^{m-1}} ,\alpha^m) = Y^m$
and $(x x^\alpha \cdots x^{\alpha^{m-1}},\alpha^m) = X^m$,
so $X^m$ and $Y^m$ are conjugate in $A$.

Conversely, suppose that $Y^m = W^{-1} X^m W$ for some $W = (w,\alpha^i)$ in~$A$,
where we choose $i\ge 0$.
Then we have
\[
y y^\alpha \cdots y^{\alpha^{m-1}} 
  = w^{-\alpha^{-i}} 
    x^{\alpha^{-i}} x^{\alpha^{1-i}}\cdots x^{\alpha^{m-1-i}}
    w^{\alpha^{m-i}}.
\]
Setting $z = x^{-\alpha^{-1}} x^{-\alpha^{-2}} \cdots x^{-\alpha^{-i}} w^{\alpha^{-i}}$,
we find that the right-hand side of the preceding equality is equal
to $z^{-1}x x^\alpha \cdots x^{\alpha^{m-1}}z^{\alpha^m},$
so the two cocycles are cohomologous.
\end{proof}

\begin{proof}[Proof of Lemma~\textup{\ref{L:conjugacy}}]
The lemma is trivial when $r=1$, so we assume throughout the proof
that $r > 1$.
To prove the first statement of the lemma for a given $X$, $Y$, and $r$,
it suffices to prove the statement for $X'=X^d$, $Y'=Y^d$, and $r'=r/d$, 
for all divisors $d$ of $r$ such that $r/d$ is a prime power.
Thus we may assume that $r$ is a power of a prime~$p$.

By replacing $Y$ with a conjugate element, we may assume that $X^r = Y^r$.
Then $X$ and $Y$ both lie in the centralizer $C_A(X^r)$ of $X^r$ in $A$,
and it will suffice to prove the lemma with $A$ replaced by $C_A(X^r)$ and
$G$ with $G\cap C_A(X^r)$.  (Note that the hypotheses of the lemma still
hold when we make these replacements.)

Now write $X=X_1 X_2$, where $X_1$ and $X_2$ are powers of $X$ such that the
order of $X_1$ is a power of $p$ and the order of $X_2$ is coprime to~$p$.
Likewise, write $Y = Y_1 Y_2$.  Since $X^r = Y^r$, we have $X_2 = Y_2$; 
furthermore, this element is a power of $X^r$ and so lies in the 
center of $A$.

We claim that $X_1^r = Y_1^r$ and that $X_1^d\not\sim Y_1^d$
for all proper divisors $d$ of $r$.  The first statement follows from the
facts that $X^r = Y^r$ and $X_2 = Y_2$.  To prove the second statement,
we note that if $X_1^d \sim Y_1^d$, we can multiply both sides of the
relation by the central element $X_2^d = Y_2^d$ to find that $X^d\sim Y^d$.

Replacing $X$ and $Y$ with $X_1$ and $Y_1$, we find that it suffices
to prove the lemma in the case where $X$ and $Y$ have $p$-power order.
Again replacing $Y$ with a conjugate, we may assume that there is a 
$p$-Sylow subgroup $S$ of $A$ that contains both $X$ and~$Y$. 
Replacing $A$ with $S$ and $G$ with $G\cap S$, 
we see that we may assume that $A$ is a $p$-group.

We prove the lemma for $p$-groups by
induction on~$\#A$.  The base case $\#A=1$ is trivial.

Since $r>1$ we know that $X$ and $Y$ are not conjugate to one another;
since they have the same image in $A/G$, the group $G$ must be nontrivial.
Every nontrivial normal subgroup of a $p$-group contains a nontrivial
central element, so there is an order-$p$ subgroup $Z$ of $G$ that 
is central in~$A$.  Let $A' = A/Z$ and $G'=G/Z$, let $x$ and $y$ be
the images of $X$ and $Y$ in $A'$, and let $s$ be the smallest divisor of $r$
such that $x^{s}\sim y^{s}$.  Replacing $Y$ by a conjugate,
we may assume that $x^{s}=y^{s}$. 
Then we have $X^{s} = c Y^{s}$ for some $c$ in $Z$, and it follows
that $s$ is either $r$ or $r/p$.  In either case, the induction
hypothesis shows that $r/p$ divides $\#G'$, so that $r$ divides~$\#G$.
\end{proof}

\begin{proof}[Proof of Lemma~\textup{\ref{L:conjugacy2}}]
We know from Lemma~\ref{L:conjugacy} that $\#G$ is even.
Suppose, to obtain a contradiction, that $\#G$ is not a multiple of $4$.
Then the odd-order elements of $G$ form an index-$2$ characteristic
subgroup $O$ of $G$; if $g$ is an element of $G$ of order $2$, 
then $G = O\cdot\langle g\rangle$.   The subgroup $O$ of $G$ is 
fixed by every automorphism of $G$, so $O$ is normal in $A$ as well.

Replacing $Y$ by a conjugate, we may assume that $X^2 = Y^2$.  Since $X$ and
$Y$ have the same image in $A/G$, we may write $Y = gX$ for some $g\in G$.  We
claim that in fact $g\in O$.  To see this, note that $A/O$ is an extension
of the order-$2$ group $G/O$ by the abelian group $A/G$, so that $A/O$ is 
abelian.  Since $X^r$ is conjugate to $Y^r$, the images of these elements
in $A/O$ are equal.  In particular, this shows that the image of $g$ in
$G/O\subset A/O$ has odd order, and so is trivial.

Now let $A' = O\cdot\langle X\rangle$ and $G' = O$.  The elements $X$ and $Y$ of $A'$
satisfy $X^2 = Y^2$ and $X\not\sim Y$, so by Lemma~\ref{L:conjugacy} we see
that $G'$ has even order, a contradiction.
\end{proof}

The following example shows that Theorem~\ref{T:cocycles} is sharp, in the sense
that for any $r$ and $s$ there are elements $x$ and $y$ of the dihedral group
$D_{2rs}$ and an automorphism $\alpha$ of $D_{2rs}$ that satisfy the hypotheses
of Theorem~\ref{T:cocycles}.

\begin{example}
Let $r>1$ and $s>1$ be two integers that are coprime to one another, and
let $G$ be the dihedral group of order $2rs$.  Choose generators $u$ and $v$
for $G$ such that $u^{rs} = v^2 = 1$ and $vuv=u^{-1}$.  Let $\alpha$
be the involution of $G$ that sends $u$ to $u^{-1}$ and $v$ to $uv$.
Let $m$ be an integer that is congruent to $0$ modulo $r$ and congruent
to $1$ modulo $s$, and take $x = uv$ and $y = u^m v$.  We claim that
$x$ and $y$ satisfy the hypotheses of Theorem~\ref{T:cocycles}.  

To see this, we first note that for any positive integer $d$ we have
\begin{align*}
x x^\alpha x^{\alpha^2}\cdots x^{\alpha^{d-1}} &= 
\begin{cases}
  u^{d/2} & \text{if $d$ is even;}\\
  u^{(d+1)/2} v & \text{if $d$ is odd;}
\end{cases}\\
y y^\alpha y^{\alpha^2}\cdots y^{\alpha^{d-1}} &= 
\begin{cases}
  u^{dm - d/2} & \text{if $d$ is even;}\\
  u^{dm - (d-1)/2} v & \text{if $d$ is odd.}
\end{cases}
\end{align*}
On the other hand, if $z = u^a$ for some integer $a$ then
\[
z^{\alpha^d} = \begin{cases}
u^a & \text{if $d$ is even;}\\
u^{-a} & \text{if $d$ is odd}
\end{cases}
\]
while if $z = u^a v$ then
\[
z^{\alpha^d} = \begin{cases}
u^a v& \text{if $d$ is even;}\\
u^{1-a} v & \text{if $d$ is odd.}
\end{cases}
\]
One can then check that then we have
\begin{equation}
\label{EQ:cocycle}
y y^\alpha y^{\alpha^2}\cdots y^{\alpha^{d-1}} =
z^{-1} x x^\alpha x^{\alpha^2}\cdots x^{\alpha^{d-1}} z^{\alpha^d}
\end{equation}
for some $z = u^a$ if and only if $u^{(m-1)d} = 1$, and 
that~\eqref{EQ:cocycle} holds for some $z = u^a v$ if
and only if $u^{md} = 1$.    The first condition holds if and only
if $r\mid d$, and the second condition holds if and only if $s\mid d$.
\end{example}

However, Theorem~\ref{T:cocycles} can be sharpened in the case where the 
automorphism $\alpha$ is the identity and $rs$ is even.
In this case, the cohomology set
$H^1(\hat{\Z},G,\alpha)$ is just the set of conjugacy classes of $G$.

\begin{theorem}
\label{T:2rs}
Let $r>1$ and $s>1$ be integers that are coprime to one another.
Suppose $G$ is a finite group that contains two elements
$x$ and $y$ such that $x^r \sim y^r$ and $x^s \sim y^s$, but such that
$x^d\not\sim y^d$ for all proper divisors $d$ of $r$ or of~$s$.
Then $\#G$ is divisible by $rs$ but is not equal to $rs$.
If $rs$ is even then $\#G$ is not equal to $2rs$, and if
$rs\equiv 2\bmod 4$ then $\#G$ is divisible by $2rs$.
\end{theorem}

\begin{proof}
We know from Theorem~\ref{T:cocycles} that $\#G$ is divisible by $rs$
but is not equal to $rs$, and we know that if $rs\equiv 2\bmod 4$ 
then $\#G$ is divisible by~$2rs$.  All we have left to prove
is that if $rs$ is even, $\#G$ cannot be equal to $2rs$.
Suppose, to get a contradiction, that $rs$ is even and  $\#G = 2rs$.
By switching $r$ and $s$, if necessary,
we may assume that $r$ is odd and $s$ is even.

The orders of $x^r$ and $y^r$ are equal, as are the orders of $x^s$ and $y^s$.
Since $r$ and $s$ are coprime, the orders of $x$ and $y$ are equal.  Denote
by $n$ this common value.  We claim that $n$ is divisible by $rs$.
To see this, suppose there were a prime power $q=p^e$ with $q\mid rs$ but $q\nmid n$,
say with $q\mid r$.  Write $x = x_1 x_2$, where $x_1$ is a power of $x$
with $p$-power order and $x_2$ is a power of $x$ with prime-to-$p$ order, and
write $y=y_1y_2$ likewise.  Our assumption is that $x_1^{q/p}= y_1^{q/p}=1$.
Since $x^r\sim y^r$ we see that
$x_2^r\sim y_2^r$.  But since the orders of $x_2$ and $y_2$ are
coprime to $p$, there is an integer $m$ with $x_2^{rm} = x_2^{r/p}$
and $y_2^{rm} = y_2^{r/p}$, so that $x_2^{r/p}\sim y_2^{r/p}$.  But since
$x_1^{r/p} = y_1^{r/p} = 1$, we have $x^{r/p}\sim y^{r/p}$ as well,
contradicting the hypotheses of the theorem.  Thus, $n$ is divisible by $rs$.

In fact, $n$ must be equal to $rs$; for if $n$ were equal to $2rs$, 
then $G$ would be abelian, and no abelian group satisfies the hypotheses
of the theorem. 

Thus the subgroup $\langle x\rangle$ of $G$ has index $2$ and is normal.
Let $z$ be an element of $G$ that is not in $\langle x\rangle$.  Then
$z x z^{-1} = x^a$ for some $a$
and $z^2 = x^b$ for some $b$.  Note that 
\[
x = z^2 x z^{-2} = z x^a z^{-1} = x^{a^2},
\]
so that $a^2 \equiv 1 \bmod rs$.  In particular, $a$ is coprime to $rs$.

We know that $x^r$ and $y^r$ are conjugate in $G$, so their images in
$G/\langle x\rangle\cong C_2$ are conjugate, hence equal.  Since $r$ is
odd, the images of $x$ and $y$ in $G/\langle x\rangle$ must be equal,
so $y$ lies in $\langle x\rangle$, say $y = x^c$.  Since $x$ and $y$ 
have the same order, $c$ must be odd.

Note that conjugating $x^i$ by an element of $G$ gives either $x^i$
or $x^{ia}$, depending on whether the element we are conjugating by is a power of $x$.
Since $x^s\sim y^s$ we see that
\[
x^{cs} = x^{s} \text{\quad or\quad} x^{cs} = x^{as}.
\]
Thus, either $c\equiv 1\bmod r$ or $c\equiv a\bmod r$.
Suppose that $c\equiv 1\bmod r$.
Then we claim that $cs/2 \equiv s/2\bmod rs$.  To check that this congruence holds,
we need only note that it holds modulo $r$ (because $c\equiv 1\bmod r$)
and that it holds modulo $s$ (because $cs/2 \equiv s/2\bmod s$,
since $c$ is odd).  Thus, $y^{s/2}$ is conjugate to $x^{s/2}$, a contradiction.

Likewise, if $c\equiv a\bmod r$, we find
that $cs/2 \equiv as/2\bmod rs$, and again $y^{s/2}$ is conjugate to $x^{s/2}$.
\end{proof}

Any group satisfying the hypotheses of Theorem~\ref{T:2rs} 
for an $s$ that is congruent to $2$ modulo $4$ must have order at
least $4rs$, but it is not necessary for its order to be a multiple
of~$4rs$.  We now construct an example for $r=3$ and $s=2$ where
the group $G$ has order $2^2\cdot 3^3$.

\begin{example}
Let $V_1$ be a $2$-dimensional vector space over $\F_2$ and let $V_2$ be
a $2$-dimensional vector space over $\F_3$.  Let $\alpha_1\in\Aut V_1$ 
and $\alpha_2\in\Aut V_2$ be automorphisms of order $3$, and let $A$ be 
the subgroup of $\Aut V_1 \times \Aut V_2$ generated by 
$\alpha = (\alpha_1,\alpha_2)$.  We take $G = (V_1\times V_2)\rtimes A$,
so that $G$ has order~$2^2\cdot 3^3$.

Let $v_1$ be a nonzero element of $V_1$ and let $v_2$ be an element of
$V_2$ that is not fixed by $\alpha_2$.  Let $x$ and $y$ be the elements
$(v_1,v_2)$ and $(\alpha_1(v_1),v_2)$ of $V_1\times V_2$, viewed as 
elements of~$G$.  It is easy to see that $x$ and $y$ are not conjugate 
in~$G$.  But $x^2$ and $y^2$ are conjugate (because they are equal), and
$x^3$ is equal to the conjugate of $y^3$ by~$\alpha$.
\end{example}

\section{The unique genus-$1$ example}
\label{S:genus1}
In this section we prove Theorem~\ref{T:genus1}.  First we make some
general comments about twists of genus-$1$ curves. 

Suppose $E$ is an elliptic curve over a perfect field~$K$.  This
means that $E$ is a curve of genus $1$ together with a specified 
$K$-rational point $O$, the identity element for the group law on~$E$.
Note that there is a distinction between an automorphism of $E$ as a 
curve and an automorphism of $E$ as an elliptic curve; a curve 
automorphism $\varphi$ of $E$ is an elliptic curve automorphism if and
only if $\varphi(O) = O$.  We denote the group of elliptic curve 
automorphisms of $E$ by $\Aut (E,O)$.

Every $\Kb$-rational point $P$ of $E$ gives us an element of $\Aut E_\Kbar$,
namely the translation-by-$P$ map, so there is an injection 
$E(\Kb) \to \Aut E_\Kbar$.  It is not hard to show that in fact we have
\[
\Aut E_\Kbar = E(\Kb) \rtimes \Aut (E_\Kbar,O),
\]
so there is a split exact sequence
\[
0 \to E(\Kb) \to \Aut E_\Kbar \to \Aut (E_\Kbar,O) \to 0.
\]
Applying Galois cohomology to this sequence, we obtain a map of pointed
sets $\pi: H^1(\Gal \Kb/K, \Aut E_\Kbar) \to H^1(\Gal \Kb/K, \Aut (E_\Kbar,O)).$
This map takes the class of a genus-$1$ curve $F$ to the class of its
Jacobian.  This fact makes it clear that $\pi$ is surjective, because
given a pointed curve $(F,P)$ that is a twist of $(E,O)$, we have 
$\pi([F]) = [(F,P)]$.  In other words, every twist of the elliptic curve
$E$ comes from a twist of the genus-$1$ curve~$E$. When $K$ is finite,
the converse is true as well.

\begin{lemma}
\label{L:ECtwists}
If $K$ is a finite field then the map $\pi$ is a bijection.
\end{lemma}

\begin{proof}
Suppose that $F_1$ and $F_2$ are genus-$1$ curves that correspond to classes
in $H^1(\Gal \Kb/K, \Aut E_\Kbar)$ having the same image in  
$H^1(\Gal \Kb/K, \Aut (E_\Kbar,O))$; that is, suppose that 
$\Jac F_1 \cong \Jac F_2$.  Deuring proved that every genus-$1$ curve
over a finite field has a rational point; in fact, the Weil bounds show 
that a genus-$1$ curve over $\F_q$ must have at least 
$q + 1 - 2\sqrt{q} \ge 1$ points.  But a genus-$1$ curve with a rational
point is isomorphic to its own Jacobian, so we must have $F_1\cong F_2$.
Thus, $\pi$ is injective.
\end{proof}

Lemma~\ref{L:ECtwists} shows that over a finite field, the twists of a
genus-$1$ curve $E$ coincide with the twists of $E$ viewed as an 
elliptic curve, so we may replace the infinite group $\Aut E_\Kbar$ with the
finite group $\Aut (E_\Kbar,O)$.

\begin{proof}[Proof of Theorem~\textup{\ref{T:genus1}}]
Suppose $E$ and $F$ are distinct elliptic curves over $K = \F_q$ that 
are twists of one another and that have minimal isomorphism extensions
of degrees $r$ and $s$.  Then Theorem~\ref{T:cocycles} shows that
the geometric automorphism group of $E$ has order divisible by $rs$
and greater than~$rs$.
Silverman~\cite{silverman}*{Prop.~A.1.2} lists the possible orders of 
automorphism groups of elliptic curves.   Using Silverman's list, we see
that either $r=2$ and $s=3$ and the elliptic curves $E$ and $F$ are twists
of the $j$-invariant $0$ curve in characteristic $2$ or in 
characteristic~$3$, or $r=3$ and $s=4$ and $E$ and $F$ are twists of the
$j$-invariant $0$ curve in characteristic~$2$. First let us examine the
cases where $r=2$ and $s=3$, starting in characteristic~$2$.

Suppose $K$ is an odd-degree extension of $\F_2$, so that $q = 2^d$ for 
some odd~$d$.  Up to isomorphism over $K$, there are exactly three 
elliptic curves over $K$ with $j$-invariant $0$; they are
\begin{align*}
E_1\colon\quad y^2 + y & = x^3 \\
E_2\colon\quad y^2 + y & = x^3 + x \\
E_3\colon\quad y^2 + y & = x^3 + x + 1.
\end{align*}
It is easy to check that these curves are pairwise nonisomorphic 
over~$K$ --- in fact, one can check that their Frobenius endomorphisms
are (respectively) $\pi_1 = (\sqrt{-2})^d$ and 
$\pi_2 = (-1+\sqrt{-1})^d$ and $\pi_3 = (1+\sqrt{-1})^d$, so they are 
not even isogenous to one another over $K$.  One can check that these 
are all of the twists of the $j=0$ curve by verifying that
$\sum_{i=1}^3 1/\#\Aut E_i = 1$ 
(see Lemma~\ref{L:weight}).  But these three
curves remain distinct over the cubic extension of~$K$; therefore, there
are no curves $C$ and $D$ over $K$ as in the statement of 
Theorem~\ref{T:genus1} when $r=2$ and $s=3$.

Suppose $K$ is an even-degree extension of~$\F_2$.  Then $C$ and $D$ are
both twists of the elliptic curve $E_1$ defined above.  All of the
geometric automorphisms of $E_1$ are defined over $K$, and 
by~\cite{silverman}*{Exer.~A.1(b)} we see that the 
automorphism group $G$ of $E_1$ is isomorphic to $\SL_2(\F_3)$.
Then by Lemma~\ref{L:cohomology}
we see that $C$ and $D$ correspond to two nonconjugate elements of $G$
whose squares and cubes are conjugate.  But an easy calculation shows 
that there are no such elements in $G$.

Suppose $K$ is an even-degree extension of~$\F_3$.  Our argument in this
case is the same as in the preceding paragraph: The curves $C$ and $D$ 
are both twists of the elliptic curve $E$ defined by $y^2 = x^3 - x$.
All of the geometric automorphisms of $E$ are defined over~$K$, and the
automorphism group $G$ of $E$ is isomorphic to $C_3\rtimes C_4$, where
$C_4$ acts on $C_3$ in the unique nontrivial way 
(see~\cite{silverman}*{Exer.~A.1(a)}).  Then by Lemma~\ref{L:cohomology}
we see that $C$ and $D$ correspond to two nonconjugate elements of $G$ 
whose squares and cubes are conjugate.  But again, an easy calculation 
shows that there are no such elements in $G$.

Thus we find that when $r=2$ and $s=3$, the field $K$ must be an 
odd-degree extension of~$\F_3$.  Now we show that for every such field 
we can find elliptic curves $C$ and $D$ as in the theorem.  

Let $a$ be an element of $K$ with $\Tr_{K/\F_3} a \neq 0$, and let $C$
and $D$ be the two curves
\begin{align*}
C\colon\quad y^2 & = x^3 - x - a\\
D\colon\quad v^2 & = u^3 - u + a.
\end{align*}
It is easy to check that these curves are not isomorphic to one another
over~$K$.  But if $i$ is an element of the quadratic extension of $K$ 
with $i^2 = -1$, then $u = -x, v = iy$ gives an isomorphism from $C$ 
to~$D$; and if $\alpha$ is an element of the cubic extension of $K$ with
$\alpha^3 - \alpha = a$, then $u = x + \alpha, v = y$ gives an 
isomorphism from $C$ to $D$.

Now let us consider the case $r=3$ and $s=4$.  Suppose there were
two genus-$1$ curves $C$ and $D$ over a finite field $K$
satisfying the conclusion of Theorem~\ref{T:genus1} with $r=3$ and $s=4$.
As we have seen, $K$ must have characteristic~$2$.
Let $L$ be the quadratic extension of $K$.  Then the curves $C_L$ and $D_L$
satisfy the conclusion of 
Theorem~\ref{T:genus1} with $r=2$ and $s=3$.
But we have just seen that no such examples exist in characteristic~$2$.
\end{proof}

\section{Proof of Theorem~\ref{T:automorphism}}
\label{S:automorphism}
In this section we prove a strong version of 
Theorem~\ref{T:automorphism}. 

\begin{theorem}
\label{T:automorphism-plus-galois}
Let $K_0$ be a finite prime field\/ $\F_p$, let $r>1$ and $s>1$ be 
integers that are coprime to one another, and let $C$ and $D$ be
curves over a finite extension $K$ of $K_0$ that satisfy the 
conclusion of Theorem~\textup{\ref{T:main}}.
Then the genus $g$ of $C$ and $D$ is larger than $0$. 
If $g=1$, then $\{r,s\} = \{2,3\}$ and the geometric automorphism 
group of $C$ \textup{(}and $D$\textup{)} viewed as an elliptic curve is 
$C_3\rtimes C_4$\textup{;}  furthermore,
not all of these automorphism are $K$-rational.
If $g>1$, then the 
geometric automorphism group $G$ of $C$ \textup{(}and~$D$\textup{)} has order 
divisible by~$rs$, but not equal to $rs$\textup{;} 
if $rs\equiv 2\bmod 4$ then $\#G$ is divisible by $2rs$\textup{;}
and if $\#G=2rs$ then no twist of $C$ has all of its geometric 
automorphisms defined over~$K$.
\end{theorem}

\begin{proof}
It follows from Theorem~\ref{T:genus0} that $g>0$, and 
the statements about the genus-$1$ case follow from 
Theorem~\ref{T:genus1} and its proof.

Suppose $g > 1$.  Let $\alpha\in\Aut G$ be the automorphism of
$G$ that describes the action of the $q$-power Frobenius on $G$,
where $q=\#K$.  As we saw in Section~\ref{S:twists}, the twist $D$
of $C$ corresponds to an element of 
$H^1(\Gal \Kb/K, G) = H^1(\hat{\Z},G,\alpha)$.
Suppose this element is represented by a cocycle that sends the
Frobenius to $y\in G$.  Then, in the notation introduced at the 
beginning of Section~\ref{S:grouptheory}, we have 
$[1]_{\alpha^r} = [y y^\alpha \cdots y^{\alpha^{r-1}}]_{\alpha^r}$
and
$[1]_{\alpha^s} = [y y^\alpha \cdots y^{\alpha^{s-1}}]_{\alpha^s}$
but for every proper divisor $d$ of $r$ or of $s$ we have
$[1]_{\alpha^d} \neq [y y^\alpha \cdots y^{\alpha^{d-1}}]_{\alpha^d}$.
Then Theorem~\ref{T:cocycles} tells us that $\#G$ is divisible by~$rs$
but not equal to $rs$, and that if $rs\equiv 2\bmod 4$ then $\#G$ is 
divisible by $2rs$.

Suppose $\#G = 2rs$, and let $X$ be an arbitrary twist of $C$.  Then 
$G \cong \Aut X_\Kbar$, and $C$ and $D$ are both twists of $X$, 
say corresponding to cocycles that send the Frobenius to elements $x$ 
and $y$ of $\Aut X_\Kbar$, respectively.  If  $\Gal \Kb/K$ 
acted trivially on $\Aut X_\Kbar$, then $x$ and $y$ would be 
nonconjugate elements of $G$ satisfying $x^r\sim y^r$ and 
$x^s\sim y^s$ and with $x^d\not\sim y^d$ for all proper divisors
$d$ of $r$ or of~$s$.  But Theorem~\ref{T:2rs} shows that this is 
impossible.  It follows that the action of  $\Gal \Kb/K$ on 
$\Aut X_\Kbar$ is nontrivial; that is, not all of the automorphisms 
of $X$ are defined over $K$.
\end{proof}

\section{Genus-$2$ examples}
\label{S:genus2}
In this section we prove Theorems~\ref{T:genus2} and~\ref{T:genus2bigger}.
In fact, we prove stronger statements that give information about the 
\emph{reduced automorphism groups} of the examples that occur for a 
given field.

First we review some facts about genus-$2$ curves. Let $C$ be a genus-$2$ curve 
over a field~$K$.  Then $C$ is hyperelliptic, so there is a unique degree-$2$
map from $C$ to $\PP^1$.  Let $\iota$ be the involution on $C$ 
determined by this double cover, and let $G$ be the automorphism group
of~$C_\Kbar$.  The subgroup $\langle\iota\rangle$ of $G$ is normal, and the
quotient group $\Gb$ is called the \emph{reduced automorphism group} 
of~$C_\Kbar$.  The group $\Gb$ acts faithfully on $\PP^1_\Kbar$ via the cover 
$C\to\PP^1$, so $\Gb$ can be viewed as a subgroup of 
$\Aut \PP^1_\Kbar = \PGL_2 \Kb$.

Let $X$ be the set of points of $\PP^1_\Kbar$ that ramify in the double
cover $C\to \PP^1$.  Then  $\Gb$ stabilizes the set $X$, and if the
characteristic of $K$ is not $2$ then every element of $\PGL_2 \Kb$ that
stabilizes $C$ is an element of  $\Gb$.

Igusa~\cite{igusa}*{\S8} enumerated the possible reduced 
automorphism groups of genus-$2$ curves over the algebraic closures of 
finite fields.  Here we determine which of these groups can occur for 
the curves mentioned in Theorem~\ref{T:genus2}.  First we dispose of the
finite fields of characteristic~$3$.

\begin{theorem}
\label{T:g2p3}
If\/ $\F_q$ is a finite field of characteristic~$3$, then there do not 
exist nonisomorphic genus-$2$ curves $C$ and $D$ over\/ $\F_q$ that
become isomorphic to one another over\/ $\F_{q^2}$ and\/ $\F_{q^3}$.
\end{theorem}

Next we classify the reduced automorphism groups that occur in 
characteristic not~$3$.

\begin{theorem}
\label{T:precise2}
Let\/ $\F_q$ be a finite field of characteristic not~$3$.  Suppose $C$ 
and $D$ are nonisomorphic genus-$2$ curves over a finite field\/ $\F_q$
that become isomorphic to one another over\/ $\F_{q^2}$ and\/ 
$\F_{q^3}$.  Let $\Gb$ be the reduced automorphism group 
of~$C_{\Fbar_q}$. 
\begin{itemize}
\item[\textup{(a)}]
If\/ $\characteristic \F_q > 5$ then the possibilities for $\Gb$ are as
follows\/\textup{:}
\begin{center}
\renewcommand{\arraystretch}{1.1}
\begin{tabular}{|c|c||c|c|c|}
\hline
\multicolumn{2}{|c||}{Given these conditions$\ldots$} &
\multicolumn{3}{|c|}{$\ldots$is this group possible\/\textup{?}}\\ \hline
$-2\in\F_q^{*2}$ \textup{?} & $-3\in\F_q^{*2}$ \textup{?} & 
$D_6$ & $D_{12}$ & $S_4$ \\
\hline\hline
\yes & \yes & \yes & \yes & \no  \\
\hline
\yes & \no  & \yes & \no  & \no  \\
\hline
\no  & \yes & \yes & \yes & \yes \\
\hline
\no  & \no  & \yes & \no  & \yes \\
\hline
\end{tabular}
\end{center}
\item[\textup{(b)}]
If\/ $\characteristic \F_q = 5$ then the possibilities for $\Gb$ are as
follows\/\textup{:}
\begin{center}
\renewcommand{\arraystretch}{1.1}
\begin{tabular}{|c||c|c|}
\hline
Given this condition$\ldots$ & 
\multicolumn{2}{|c|}{$\ldots$is this group possible\/\textup{?}}\\ \hline
$-2\in\F_q^{*2}$ \textup{?} & $D_6$ & $S_5$ \\
\hline\hline
\yes & \yes & \yes \\
\hline
\no  & \yes & \no  \\
\hline
\end{tabular}
\end{center}
\item[\textup{(c)}]
If\/ $\characteristic \F_q = 2$ then $\Gb \cong D_6$.
\end{itemize}
Furthermore, all of the groups listed as possibilities for a given 
field\/ $\F_q$ actually do occur.  In particular, over every finite
field of characteristic not~$3$, there are examples where $C$ and $D$
have geometric reduced automorphism group isomorphic to~$D_6$.
\end{theorem}

First we will prove Theorem~\ref{T:precise2}, then Theorem~\ref{T:g2p3},
and finally Theorem~\ref{T:genus2bigger}.

\begin{remark}
On three occasions in the following proof of Theorem~\ref{T:precise2}
we will have to show that if $x$ and $y$ are elements of a certain
group $G$ and if $x^2\sim y^2$ and $x^3\sim y^3$, then $x\sim y$.  
On the first occasion we will give the details of the computation.
On the second occasion we will leave the details to the reader.
On the third occasion we will translate the group-theoretic
problem into a question about explicit curves over a small finite
field, and then answer this question by direct computation.
\end{remark}

\begin{proof}[Proof of Theorem~\textup{\ref{T:precise2}}]
Suppose that the characteristic of $\F_q$ is greater than~$5$.  In this 
case, Igusa~\cite{igusa}*{\S8} calculates that there are $7$ 
possible reduced isomorphism groups for a genus-$2$ curve, each of which
actually occurs over the algebraic closure:  the trivial group, the
cyclic group $C_2$, the dihedral group $D_6$, the Klein $4$-group 
$V_4 = D_4$, the dihedral group $D_{12}$, the symmetric group $S_4$, and
the cyclic group~$C_5$.

Since $C$ and $D$ are nontrivial twists of one another 
over the quadratic and cubic extensions of $\F_q$, 
Theorem~\ref{T:automorphism} says that the geometric automorphism groups
of $C$ and $D$ must have order divisible by~$12$.  We see immediately
that the only possible reduced geometric automorphism groups are $D_6$, 
$D_{12}$, and~$S_4$.  To prove statement (a), we must show that $S_4$ 
cannot occur when $-2$ is a square in $\F_q$ and that $D_{12}$ cannot
 occur when $-3$ is not a square in $\F_q$.

Igusa shows that in fact there is exactly one genus-$2$ curve over 
$\Fbar_q$ whose reduced automorphism group is $S_4$.  This curve can 
always be defined over $\F_q$; one model for it is 
\[
y^2 = x^6 - 5x^4 - 5x^2 + 1.
\]
Call this model $X$, and let $a$ be a square root of $-2$ in $\Fbar_q$.
The geometric automorphism group $G$ of $X$ is generated by the 
automorphisms $\alpha$, $\beta$, $\gamma$, and $\delta$ defined by
\begin{align*}
\alpha(x,y) &= \left( \frac{x+1-a}{(-a-1)x+1}, 
                       \frac{8y}{((a+1)x-1)^3} \right)\\
\beta(x,y)  &= \left( \frac{x+1}{x-1}, \frac{2ay}{(x-1)^3} \right)\\
\gamma(x,y) &= (-x,y)\\
\delta(x,y) &= (1/x,y/x^3).
\end{align*}
The orders of these automorphisms are $3$, $4$, $2$, and $2$, 
respectively, and $\beta^2$ is the hyperelliptic involution~$\iota$.  Clearly
all of the geometric automorphisms of $X$ are defined over $\F_q$ if 
$-2$ is a square in $\F_q$; otherwise, there are only $8$ elements of 
$G$ defined over $\F_q$, namely the subgroup generated by $\beta^2$, 
$\gamma$, and $\delta$.  Note that no element of $G$ has order $12$, 
because no element of $\Gb\cong S_4$ has order $6$.

Now suppose that $-2$ is a square in $\F_q$, and suppose that we have
curves $C$ and $D$ over $\F_q$ that are quadratic and cubic twists of
one another and that have reduced geometric automorphism groups 
isomorphic to $S_4$.  Then $C$ and $D$ must both be $\Fbar_q$-twists 
of~$X$.  By Lemma~\ref{L:cohomology}, the curves $C$ and $D$ correspond
to conjugacy classes in $G$, say the conjugacy classes of elements $u$ 
and $v$ respectively.  Our assumptions on $C$ and $D$ imply that 
$u^2\sim v^2$ and $u^3\sim v^3$.  Since $u^2$ and $v^2$ have the same
order, as do $u^3$ and $v^3$, we see that $u$ and $v$ have the same
order.  If this order is a power of $2$ then $u^3\sim v^3$ implies that
$u\sim v$.  If this order is $3$ then $u^2\sim v^2$ implies that 
$u\sim v$.  The only other possibility is that this order is $6$. 
Since $\Gb\cong S_4$ has no elements of order $6$, if $u$ and $v$ have
order $6$ then $u^3 = v^3 = \iota$.  Also, $S_4$ has only one conjugacy
class of elements of order $3$, so if $u$ and $v$ have
order $6$ then either $u\sim v$ or $u\sim \iota v = v^4$.  But
$v^4$ has order $3$, so we must have $u\sim v$.
In every case we have $u\sim v$, so $C$ and $D$ must be
isomorphic to one another over $\F_q$.   Thus, $S_4$ is not a
possibility when $-2$ is a square in $\F_q$.

Igusa also shows that there is exactly one genus-$2$ curve over
$\Fbar_q$ whose reduced automorphism group is $D_{12}$.  This curve can
always be defined over~$\F_q$; one model for it is 
\[
y^2 = x^6 + 1.
\]
Call this model $X$, and let $\omega$ be a primitive cube root of $1$ in
$\Fbar_q$.  The geometric automorphism group $G$ of $X$ is generated by
the automorphisms $\alpha$, $\beta$, $\gamma$, and $\iota$ defined by
\begin{align*}
\alpha(x,y) &= (\omega x, y)\\               
\beta(x,y)  &= \left( \frac{1}{x}, \frac{y}{x^3} \right)\\
\gamma(x,y) &= (-x,y)\\
\iota(x,y) &= (x,-y).
\end{align*}
The orders of these automorphisms are $3$, $2$, $2$, and $2$, 
respectively.  Clearly all of the geometric automorphisms of $X$ are
defined over $\F_q$ if $-3$ is a square in $\F_q$; otherwise, there are
only $8$ elements of $G$ defined over $\F_q$, namely the subgroup 
generated by $\beta$, $\gamma$, and $\iota$. 

Now suppose that $-3$ is not a square in $\F_q$.   If $C$ and $D$ are
curves over $\F_q$ that are quadratic and cubic twists of one another
and that have reduced geometric automorphism groups isomorphic 
to~$D_{12}$, then $C$ and $D$ are both $\Fbar_q$-twists of~$X$.  Let $c$
and $d$ be cocycles that represent the classes in 
$H^1(\Gal \Fbar_q/\F_q ,G)$ giving rise to $C$ and $D$, and let $u$ and
$v$ be the images of Frobenius under the cocycles $c$ and $d$, 
respectively.  

Let $\phi$ be the element of $\Aut G$ that gives the action of the
Frobenius $\varphi \in \Gal \Fbar_q/\F_q$ on~$G$, and let 
$A = G \rtimes \langle\phi\rangle$.  Using Lemma~\ref{L:semidirect},
we see that the elements $\tilde{u}=(u,\phi)$ and $\tilde{v}=(v,\phi)$
of $A$ have conjugate squares and conjugate cubes.

Using the explicit description of $G$ given above, it is a straightforward
matter to show that $\tilde{u}$ and $\tilde{v}$ are conjugate to one another.
(One can simply enumerate all pairs of elements of $A$ whose squares and
cubes are conjugate to one another, and verify that the elements themselves
are conjugate to one another.)  By  Lemma~\ref{L:semidirect} we see that 
the cocycles $c$ and $d$ are cohomologous, so the curves $C$ and $D$ are
isomorphic to one another over $\F_q$.

Thus we see that when $-3$ is not a square in $\F_q$, there do not exist
nonisomorphic curves $C$ and $D$ over $\F_q$ that become isomorphic to
one another over $\F_{q^2}$ and $\F_{q^3}$ and that have with geometric
reduced automorphism groups isomorphic to $D_{12}$.

This completes the proof of statement (a).

Suppose that the characteristic of $\F_q$ is~$5$.  Then Igusa determines
that the possible reduced automorphism groups are the trivial group, 
$C_2$, $D_6$, $D_4$, and $\PGL_2(\F_5)\cong S_5$.  
Theorem~\ref{T:automorphism} shows that the only possible reduced 
automorphism groups for curves $C$ and $D$ as in the statement of 
Theorem~\ref{T:genus2} are $D_6$ and $S_5$.  To complete the proof of 
statement~(b) we must show that $S_5$ cannot occur when $q$ is an odd
power of~$5$.

Let $X$ be the curve $y^2 = x^5 - x$.  The double cover $X\to\PP^1$
is ramified at $\PP^1(\F_5)$, so the reduced automorphism group of $X$
is $\PGL_2(\F_5)\cong S_5$.  The reduced automorphisms are defined 
over~$\F_5$, but the automorphisms of $X$ are not all defined 
over~$\F_5$:  lifting the reduced automorphism $x\mapsto -2/x$ requires
a square root of~$-2$.

Igusa shows that up to geometric isomorphism, $X$ is the only genus-$2$
curve in characteristic $5$ with reduced automorphism group~$S_5$, so if
there are examples of pairs of curves $C$ and $D$ as in the statement of
the theorem that have geometric reduced automorphism group $S_5$, they 
will have to be twists of~$X$.  Our work in Section~\ref{S:twists} shows
that the existence of such $C$ and $D$ is really a question about
$H^1(\Gal \Fbar_q/\F_q,\Aut X_{\Fbar_q})$.  Now, this cohomology set
\emph{does not depend on $q$}, so long as $q$ is an odd power of~$5$.
Thus, to show that no $C$ and $D$ exist for an arbitrary odd power 
of~$5$, it will suffice to show that there are no such $C$ and $D$ 
over~$\F_5$.  

One can easily check that the curves $y^2 = f$ over~$\F_5$, with $f$
taken from the set
\[
   \{x^5 - x, x^5 - 2x, x^5 + x, x^5 - x + 1, x^5 - x + 2,
    x^6 - 2, x^6 - x + 1, x^6 - x + 2\},
\]    
are all twists of $X$ and are all pairwise nonisomorphic.  It is also
easy to check these curves cover all $\F_5$-isomorphism classes of
twists of $X$, because the sum of the inverses of the cardinalities of
their groups of rational automorphisms is equal to $1$ 
(see Lemma~\ref{L:weight}).  Then one can verify that no two
of these curves become isomorphic to one another over both $\F_{25}$ 
and~$\F_{125}$.  (This last computation is not hard; simply by looking 
at the fields of rationality of the Weierstrass points of the curves, 
one finds that the only possible $\{C,D\}$ pair is 
\[
\{y^2 = x^5 - x + 1, y^2 = x^5 - x + 2\},
\]
and these curves have different numbers of points over $\F_{125}$.)
This proves statement~(b).

Suppose that the characteristic of $\F_q$ is~$2$.  Again Igusa lists the
possible reduced automorphism groups, and the only group on the list 
whose order is divisible by $3$ is the group $D_6$.  This proves 
statement~(c).

The fact that the groups that we have not excluded actually
do occur follows from Examples~\ref{X:D6}--\ref{X:S5} below.
\end{proof}

\begin{example}
\label{X:D6}
\emph{The group $D_6$ in characteristic greater than~$3$.}

Suppose the characteristic of $\F_q$ is greater than~$3$.  We first show
that there is a genus-$2$ curve over $\F_q$ whose reduced automorphism
group is $D_6$ and that has automorphisms that are not defined
over~$\F_q$.

Suppose $t \in \F_q$ satisfies
\begin{equation}
\label{EQ:forbidden-t1}
t
(t-1)
(t+1)
(t^2-t+1)\neq 0.
\end{equation}
Let $n = t^2 - t + 1$, let $a_4 = -3(t^4+t^2+1)/t$, let
$a_3 = 2(3t^5 - 2t^4 + 3t^3 + 3t^2 - 2t + 3)/t$, and let $f$ be the 
polynomial
\[
f = x^6 + a_4 x^4 + a_3 x^3 + n a_4 x^2 + n^3.
\]
One can check that Inequality~\eqref{EQ:forbidden-t1} implies that the 
discriminant of $f$ is nonzero.  Furthermore, one can check that the 
automorphisms $\bar{\alpha}(x) = n/x$ and 
$\bar{\beta}(x) = (tx-n)/(x-1)$ of $\PP^1$ permute the roots of~$f$.
The automorphisms $\bar{\alpha}$ and $\bar{\beta}$ generate a group 
isomorphic to $D_6$, so the reduced automorphism group of the genus-$2$
curve $X$ defined by $y^2 = f$ contains~$D_6$.

Igusa~\cite{igusa}*{\S8} shows that for odd prime powers $q$ 
there are at most two curves over $\Fbar_q$ whose reduced automorphism 
groups contain $D_6$ and are larger than $D_6$; one can compute that the
Igusa invariants of these curves are
\[
[20{\col}30{\col}{-20}{\col}{-325}{\col}64]
\]
and 
\[
[120{\col}330{\col}{-320}{\col}{-36825}{\col}11664].
\]
One can show that as long as 
\begin{equation}
\label{EQ:forbidden-t2}
(t^2 + t + 1)
(3t^2 - 2t + 3)
(t^2 - 4t + 1)
\neq 0
\end{equation}
the curve $X$ does not have these Igusa invariants, so in this case its
reduced automorphism group is exactly $D_6$.

The automorphism $\bar{\beta}$ of $\PP^1$ automatically lifts to an
$\F_q$-defined automorphism of $X$, but $\bar{\alpha}$ will only lift to
an $\F_q$-defined automorphism if $n$ is a square in $\F_q$.  We will 
show that when $q\ne 7$ we can find a $t\in\F_q$ that satisfies 
Inequalities~\eqref{EQ:forbidden-t1} and~\eqref{EQ:forbidden-t2} and such
that $n$ is not a square.

Let $n_0$ be an arbitrary nonsquare in $\F_q$.  The curve 
$x^2 - x + 1 - n_0 y^2 = 0$ over $\F_q$ is a nonsingular conic, so it 
has at least $q-1$ rational points in the affine plane.  At most $2$ of
these points have $y$-co\"ordinate equal to $0$, so there are at least
$(q-3)/2$ values of $t$ in $\F_q$ such that $t^2 - t + 1$ is equal to a 
nonsquare in $\F_q$.   At most one of these values (namely $t=-1$)
fails to satisfy Inequality~\eqref{EQ:forbidden-t1}, 
because the other values of $t$ that fail to satisfy
Inequality~\eqref{EQ:forbidden-t1} do not have the property that 
$t^2-t+1$ is a nonsquare.  At most six values of $t$ fail to 
satisfy Inequality~\eqref{EQ:forbidden-t2}.  Thus, as
long as $(q-3)/2 > 7$ we are assured that there is a value of $t$ in 
$\F_q$ such that the curve $X$ constructed above has reduced 
automorphism group equal to $D_6$ and has automorphisms that are not 
defined over $\F_q$.

For the primes $5$, $11$, $13$, and $17$, the value $t=-8$ satisfies 
Inequalities~\eqref{EQ:forbidden-t1} and~\eqref{EQ:forbidden-t2}, and 
$t^2 - t + 1 = 73$ is a nonsquare modulo these primes.

For $q = 7$, we consider the curve $X$ defined by 
\[
y^2 = x^6 + x^5 - 3x^4 - 2x^2 + 2x - 1,
\]
whose Igusa invariants are \hbox{$[ 0: 4: 4: 3: 4 ]$}. The reduced
automorphism group of $X$ contains the automorphisms 
$\bar{\alpha}(x) = 3/x$ and $\bar{\beta}(x) = (2x-3)/(x-1)$, which 
generate a group isomorphic to $D_6$.  The reduced automorphism group is
no larger than this because the Igusa invariants of $X$ are not equal to
those of the two curves with larger groups, and since $3$ is not a 
square in $\F_7$ the automorphism $\bar{\alpha}$ does not lift to an
$\F_7$-rational automorphism of~$X$.

Thus for every $\F_q$ of characteristic greater than $3$ we know there 
is a genus-$2$ curve $X$ over $\F_q$ whose reduced automorphism group
is isomorphic to $D_6$ and is generated by an automorphism 
$\bar{\alpha}$ of order $2$ that does not lift to a rational 
automorphism of $X$ and an automorphism $\bar{\beta}$ of order $3$ that
does lift rationally.

Let $\alpha$ and $\beta$ be lifts of $\bar{\alpha}$ and $\bar{\beta}$ to
$G = \Aut X_{\Fbar_q}$.  Note that $\alpha$ and $\beta$ generate a group
isomorphic to $D_6$, and $G$ is the product of this group with the
order-$2$ subgroup containing the hyperelliptic involution $\iota$.  Let
$\varphi \in \Gal \Fbar_q/\F_q$ be the $q$-power Frobenius, so that
$\alpha^\varphi = \iota \alpha$.

Let $c\in H^1(\Gal \Fbar_q/\F_q, G)$ be the cocycle that sends
$\varphi$ to $\beta$, and let $d$ be the cocycle that sends $\varphi$ to
$\beta^2$.  We claim that $c$ and $d$ are not cohomologous.  To verify
this, we must show that there is no $\gamma\in G$ with
$\beta^2 = \gamma^{-1} \beta \gamma^\varphi.$  Suppose we take an 
arbitrary $\gamma$ in $G$ and write it as 
$\gamma = \beta^i \alpha^j \iota^k$.  Then 
$\gamma^\varphi = \beta^i \alpha^j \iota^{j+k}$, so
\begin{align*}
\gamma^{-1} \beta \gamma^\varphi 
 &= (\iota^{-k} \alpha^{-j} \beta^{-i}) \beta 
                                       (\beta^i \alpha^j \iota^{j+k})\\
 &= \alpha^{-j} \beta  \alpha^j \iota^j\\
 &= \begin{cases}
      \beta          & \text{\quad if $j=0$;}\\
      \iota \beta^2  & \text{\quad if $j=1$.}
    \end{cases}
\end{align*}
Thus $c$ and $d$ are not cohomologous.

Consider the cocycles in  $H^1(\Gal \Fbar_q/\F_{q^2}, G)$ induced from
$c$ and $d$; these are the cocycles that send $\varphi^2$ to
$\beta \beta^\varphi = \beta^2$ and to
$(\beta^2)(\beta^2)^\varphi = \beta$, respectively.  Since 
$\beta^2 = \alpha^{-1} \beta \alpha^{\varphi^2}$, we see that these
cocycles are cohomologous to one another.

Next consider the cocycles in  $H^1(\Gal \Fbar_q/\F_{q^3}, G)$ induced
from $c$ and $d$; these are the cocycles that send $\varphi^3$ to
$\beta \beta^\varphi \beta^{\varphi^2} = 1$ and 
$(\beta^2)(\beta^2)^\varphi(\beta^2)^{\varphi^2} = 1$.  Clearly these 
are cohomologous to one another.

It follows that if we let $C$ and $D$ be the twists of $X$ corresponding
to the cohomology classes of $c$ and $d$, then $C$ and $D$ are not
isomorphic to one another, but become isomorphic to one another over the
quadratic and cubic extensions of $\F_q$.  Furthermore, their geometric
reduced automorphism groups are isomorphic to $D_6$.
\end{example}

\begin{example}
\label{X:D6-char2}
\emph{The group $D_6$ in characteristic~$2$.}

Suppose $\F_q$ is a finite field of characteristic~$2$. Let $a\in \F_q$
be an element whose trace to $\F_2$ is $1$, and let $b\in \F_{q^2}$ be
an element with $b^2 + b = a$.  Note that $b\not\in\F_q$, because $a$ 
has trace~$1$.

Let $X$ be the curve 
\[
y^2 + y = a\left(x + \frac{1}{x} + \frac{1}{x+1}\right).
\]
The reduced automorphism group of $X$ contains the automorphisms 
$\bar{\alpha}(x) = x + 1$ of order $2$ and $\bar{\beta}(x) = 1/(x+1)$ of
order~$3$, and Igusa shows that the $D_6$ generated by these 
automorphisms is the full reduced automorphism group of~$X$.  Note that
$\bar{\beta}$ lifts to the $\F_q$-rational automorphism 
\[
(x,y) \mapsto \left(\frac{1}{x+1},y\right)
\]
of $X$, but that  $\bar{\alpha}$ only lifts to the automorphisms 
$(x,y)\mapsto(x+1,y+b)$ and $(x,y)\mapsto(x+1,y+b+1)$, which are not
defined over~$\F_q$.

Now we are in the same group-theoretical situation as we were in
Example~\ref{X:D6}, and the same argument shows that we can find 
nonisomorphic twists $C$ and $D$ of $X$ that become isomorphic to one 
another over $\F_{q^2}$ and $\F_{q^3}$.
\end{example}

\begin{example}
\label{X:D12}
\emph{The group $D_{12}$ in characteristic greater than~$5$, when
$-3$ is a square.}

Suppose $\F_q$ has characteristic greater than $5$ and suppose $-3$ is a
square in $\F_q$, so that $q\equiv 1 \bmod 3$.  Let $g$ be a generator 
for $\F_q^*$.  An argument using Lemma~\ref{L:kummer} shows that 
the two curves $y^2 = x^6 + g$ and $gy^2 = x^6 + g$ over $\F_q$ are not isomorphic 
to one another, but $(x,y)\mapsto(x,y/g^{1/2})$ gives an isomorphism over $\F_{q^2}$,
and $(x,y)\mapsto(g^{1/3}/x, y/x^3)$ gives an isomorphism over $\F_{q^3}$.
Furthermore, the two curves
are geometrically isomorphic to the curve $y^2 = x^6 + 1$, which in
characteristic greater than $5$ has reduced automorphism group $D_{12}$.
\end{example}

\begin{example}
\label{X:S4}
\emph{The group $S_4$ in characteristic greater than~$5$, when $-2$ is
not a square.}

Suppose $\F_q$ has characteristic greater than $5$ and suppose $-2$ is 
not a square in $\F_q$.  Let $X$ be the curve 
\[
y^2 = x^5 - x
\]
over $\F_q$.  Then the geometric reduced automorphism group of $X$ is
isomorphic to $S_4$.  (Perhaps more suggestively, we can say that the 
geometric reduced automorphism group is isomorphic to the octahedral 
group, the octahedron in question being the one in $\PP^1_\BC$ whose
vertices are the roots of $z^5-z$ in $\BC$, together with $\infty$.)

Let $\zeta$ be a primitive $8$th root of unity in $\F_{q^2}$ and let
$i = \zeta^2$.  The fact that $-2$ is not a square in 
$\F_q$ implies that $q$-th power raising sends $\zeta$ to either 
$\zeta^5$ or $\zeta^7$.  Then the reduced automorphism group of $X$,
viewed as a subgroup of $\PGL_2(\Fbar_q)$, consists of the elements
\[
  \left\{ i^a z^b \ \big|\  a\in \{0,1,2,3\}, b\in\{-1,1\}, 
               z\in \left\{x, \frac{x-1}{x+1}, \frac{x-i}{x+i}\right\} 
  \right\}.
\]  

Let $c$ be the cocycle in $H^1(\Gal \Fbar_q/\F_q,G)$ that sends the
Frobenius $\varphi\in\Gal \Fbar_q/\F_q$ to the automorphism
\[
\alpha: (x,y) \mapsto \left( \frac{x + i}{x - i},
                       \frac{2+2i}{(x - i)^3}\cdot y\right),
\]                       
and let $d$ be the cocycle that sends $\varphi$ to
\[
\beta: (x,y) \mapsto \left( \frac{ix - i}{x + 1},
                       \frac{2-2i}{(x + 1)^3}\cdot y\right).
\]
One can check that no matter whether $\varphi(\zeta)$ is $\zeta^5$ or
$\zeta^7$, these two cocycles are not cohomologous.  However, their 
images in  $H^1(\Gal \Fbar_q/\F_{q^2},G)$ \emph{are} cohomologous; 
indeed, if we let $\gamma$ be the automorphism
\[
\gamma: (x,y) \mapsto   \left( \frac{x - 1}{x + 1},
                       \frac{\zeta(2-2i)}{(x + 1)^3}\cdot y\right),
\]
then we have
\[
\alpha \alpha^\varphi = 
  \gamma^{-1} \beta \beta^\varphi \gamma^{\varphi^2}.
\]  
Also, the images of $c$ and $d$ in $H^1(\Gal \Fbar_q/\F_{q^3},G)$ are
also cohomologous, and in fact
\[
\alpha \alpha^\varphi \alpha^{\varphi^2} = 
   \beta  \beta^\varphi  \beta^{\varphi^2}.
\]

Therefore, the twists of $X$ corresponding to $c$ and $d$ give us the
example we want.
\end{example}

\begin{example}
\label{X:S5}
\emph{The group $S_5$ in characteristic~$5$, when $-2$ is a square.}

We are working over $\F_q$, where $q$ is an even power of $5$.   Let $g$ be a
generator of the multiplicative group $\F_q^*$, and consider
the two curves $y^2 = x^6 + g$ and $gy^2 = x^6 + g$.  The same argument
we gave in Example~\ref{X:D12} shows that these two curves are not isomorphic 
to one another, but that they become isomorphic over $\F_{q^2}$ and $\F_{q^3}$.
Since they are both twists of $y^2 = x^5 - x$, they have geometric reduced 
isomorphism groups isomorphic to $S_5.$
\end{example}

Next we prove Theorem~\ref{T:g2p3}.

\begin{proof}[Proof of Theorem~\textup{\ref{T:g2p3}}]
Igusa shows that the possible reduced automorphism groups of genus-$2$
curves in characteristic~$3$ are the trivial group, $C_2$, $D_6$, $D_4$,
$S_4$, and~$C_5$.  Theorem~\ref{T:automorphism} shows that the only 
possibilities in our situation are $D_6$ and $S_4$.  But our proof of
statement~(a) of Theorem~\ref{T:precise2} shows that $S_4$ is not a 
possibility, so all we have to show is that $D_6$ is also impossible.
By Theorem~\ref{T:automorphism-plus-galois}, it will be enough for us to
show that every curve $C$ over a finite field $\F_q$ of characteristic 
$3$ whose reduced automorphism group is $D_6$ has a twist that has all 
of its automorphisms defined over~$\F_q$.

So let $q$ be a power of $3$ and let $C$ be a curve over $\F_q$ with 
geometric reduced automorphism group $D_6$.  From Igusa's 
parametrization of such curves (over~$\Fbar_q$) one can calculate that 
the Igusa invariants of such a curve are 
\hbox{$[1{\col}0{\col}t{\col}t{\col}t]$} for some nonzero $t\in \F_q$.
(Here we are using the invariants 
\hbox{$[J_2{\col}J_4{\col}J_6{\col}J_8{\col}J_{10}]$} from~\cite{igusa},
so the vector of Igusa invariants lives in weighted projective space, 
where the $i$-th co\"ordinate has weight $i$.)  In fact, we also have
$t\neq -1$, because that value of $t$ comes from the curve with reduced
automorphism group $S_4$.  So it will suffice for us to show that given 
any $t\in \F_q$ with $t\neq0$ and $t\neq -1$, there is a curve with 
Igusa invariants \hbox{$[1{\col}0{\col}t{\col}t{\col}t]$} defined over
$\F_q$ with all of its automorphisms defined over~$\F_q$.

Let $X$ be the curve $y^2 = (x^3-x)^2 - t^{1/3}$.  One can compute that
the Igusa invariants of $X$ are 
\hbox{$[1{\col}0{\col}t{\col}t{\col}t]$}.  But for every $a\in \F_3$, we
 have automorphisms
\[
(x,y) \mapsto (\pm x + a, \pm y)
\]
of $X$, so all $12$ of the  automorphisms of $X$ are defined 
over~$\F_q$.
\end{proof}

We close by proving Theorem~\ref{T:genus2bigger}.

\begin{proof}[Proof of Theorem~\textup{\ref{T:genus2bigger}}]
Let  $r$, $s$, $C$, and $D$ be as in the statement of 
Theorem~\ref{T:genus2bigger}.  From  Theorem~\ref{T:automorphism} we know
that $rs$ divides the order of the geometric automorphism
groups of $C$ and $D$, but that this group has order larger than $rs$.
According to Igusa's enumeration~\cite{igusa}*{\S8} of the 
possible geometric reduced automorphism groups of genus-$2$ curves,
we find that there are only four possible reduced
automorphism groups to consider:
\begin{enumerate}
\item The group $D_{12}$, which occurs in characteristic larger than $5$
      as the reduced automorphism group of the curve $y^2 = x^6 + 1$,
      and which occurs in no other way;
\item The group $S_4$, which occurs in characteristics other than $2$ and $5$
      as the reduced automorphism group of the curve $y^2 = x^6 - 5x^4 - 5x^2 + 1$,
      and which occurs in no other way;
\item The group $S_5$, which occurs in characteristic $5$ as the
      reduced automorphism group of the curve $y^2 = x^5 - x$, and
      which occurs in no other way; and
\item A group of the form $C_2^4\rtimes C_5$, which occurs in characteristic $2$
      as the reduced automorphism group of the curve $y^2 + y = x^5$,
      and which occurs in no other way.
\end{enumerate}
Note that the equations we give for these four curves show that they can be defined over
the appropriate prime field $\F_p$; also the full automorphism groups of the first three
curves can be defined over, at worst, the quadratic extension of the ground field. 
(For the fourth curve, the full automorphism group may require a quadratic or a quartic
extension of the ground field.)

For any given $\{r,s\}$ pair and any finite base field $K$, Theorem~\ref{T:cocycles}
says that to see whether we can find two twists of one of these four curves that
satisfy the conclusion of Theorem~\ref{T:main}, we must make simple computations 
in the cohomology sets $H^1(\hat{\Z},G,\alpha^d)$ for various values of $d$, 
where $G$ is the automorphism group of the given curve and $\alpha$ is the 
automorphism of $G$ induced by Frobenius.  But we know what the possible 
groups $G$ are, and we know what the possible Frobenius actions are, and we
know which $\{r,s\}$ pairs we have to consider, so determining whether any 
$\{r,s\}$ pairs give rise to examples is a finite and well-defined computation.

We have provided examples of similar computations in the proof of Theorem~\ref{T:genus2}.
We leave the ones required here to the reader (and his or her favorite symbolic
manipulation program).  But we will at least note that the computation for $S_4$ can be 
skipped: for this group, the only possibilities for $\{r,s\}$ are $\{3,4\}$ and $\{3,8\}$,
and if $C$ and $D$ over $\F_q$ give an example for such a pair, then $C$ and $D$, when
base extended to $\F_{q^2}$ or $\F_{q^4}$, given an example with $\{r,s\} = \{2,3\}$ 
over a field of square order.  However, Theorem~\ref{T:genus2} shows that no
such examples exist.

The result of the computation that we have outlined above is that the only reduced
automorphism group that gives rise to an example is the group $S_5$, with the
nontrivial action of Frobenius, and with $\{r,s\} = \{2,5\}$.
Furthermore, there is a unique pair of elements of $H^1(\hat{\Z},G,\alpha)$
that gives rise to an example: identifying $G$ with the automorphism group
of $y^2 = x^5 - x$, the pair corresponds to the classes in $H^1(\hat{\Z},G,\alpha)$
of the automorphisms $(x,y)\mapsto(x+1,y)$ and $(x,y)\mapsto(x+2,y)$.

Since the action of Frobenius on the automorphism group is nontrivial, the
examples occur only over fields of the form $K = \F_q$ where $q$ is an odd power of~$5$.
Let $\sigma$ be the $q$-power Frobenius of $\Kb$.  The first of the cohomology
classes identified above corresponds to the twists of $y^2 = x^5 - x$ of the form
$y^2 = x^5 - x + a$, where $a\in K$ has the property that $b^\sigma - b = \pm1$
for all $b\in \Kb$ with $b^5 - b = a$.  The second cohomology class corresponds
to the twists of the form $y^2 = x^5 - x + a$, where  $a\in K$ has the property 
that $b^\sigma - b = \pm2$ for all $b\in \Kb$ with $b^5 - b = a$.
Thus, we can always put the twists $C$ and $D$ in the form given in 
the statement of the theorem, and any two curves as in the statement of
the theorem are twists of $y^2 = x^5 - x$ with minimal isomorphism extensions
of degrees $2$ and $5$ over $K$.
\end{proof}

\section{Galois cohomology of connected algebraic groups}
\label{S:related}

Suppose $C$ and $D$ are two curves over a field $K$.  Let $X$ be the 
algebraic variety $\Isom(C,D)$, and let $L$ and $M$ be finite
extensions of~$K$ whose degrees over $K$ are coprime to one another.
Question~\ref{Q:oldmain} can be phrased in terms of $X$:
\begin{question}
\label{Q:general}
Suppose $X(L)$ and $X(M)$ are nonempty.  Must $X(K)$ be nonempty\/\textup{?}
\end{question}
Note that the hypothesis that $X(L)$ is nonempty shows that
the variety $X$ is a torsor for the algebraic group~$\Aut C$.

Several authors have considered Question~\ref{Q:general} in the
case where $X$ is a torsor for a connected linear algebraic group 
(see~\cite{totaro} and the references therein).  
Totaro considers a more
general version of the question, which he phrases as an assertion:

\begin{quotation}
\noindent
Let $k$ be a field, let $G$ be a smooth connected linear algebraic group
over~$k$, and let $X$ be a quasi-projective variety that is a
homogeneous space for $G$.  Suppose that there is a zero-cycle (not
necessarily effective) of degree $d > 0$ on $X$.  Then $X$ has a closed
point of degree dividing $d$, which moreover can be chosen to be \'etale
(i.e.~separable) over $k$.~\cite{totaro}*{Question~0.2}
\end{quotation}

Totaro points out that the question has a positive answer in the special case 
$d=1$ for torsors over any split simple group other than $E_8$, by work of 
Bayer-Fluckiger and Lenstra~\cite{bayer-fluckiger-lenstra} and
Gille~\cite{gille};  Garibaldi and Hoffmann~\cite{garibaldi-hoffmann} proved 
that the answer is again positive for torsors over several other groups, 
including some non-split ones.  The answer for $E_8$ is not known. 
Florence~\cite{florence} and Parimala~\cite{parimala} have shown that the 
answer can be `no' if $X$ is not a $G$-torsor.

Serre makes the following remark:
\begin{quotation}
\noindent
Soit $G$ un groupe alg\'ebrique sur $k$, et soient $x,y$ deux
\'el\'ements de $H^1(k,G)$.  Supposons que $x$ et $y$ aient m\^eme
images dans $H^1(k',G)$ et dans $H^1(k'',G)$ o\`u $k'$ et $k''$ sont 
deux extensions finies de $k$ de degr\'es premiers entre eux (par 
exemple $\indexin{k'}{k} = 2$ et $\indexin{k''}{k} = 3$).
\emph{Ceci n'entra\^{\i}ne pas} $x = y$ contrairement \`a ce qui se 
passe dans le cas ab\'elien ; on peut en construire des exemples, en 
prenant $G$ non connexe ; j'ignore ce qu'il en est lorsque $G$ est 
connexe.~\cite{serre:RC9091}*{p.~117}
\end{quotation}
\begin{quotation}
\noindent
[Let $G$ be an algebraic group over $k$, and let $x$ and $y$ be two 
elements of $H^1(k,G)$.  Suppose that $x$ and $y$ have the same images 
in $H^1(k',G)$ and in $H^1(k'',G)$, where $k'$ and $k''$ are finite 
extensions of $k$ whose degrees are coprime to one another (for example,
$\indexin{k'}{k} = 2$ and $\indexin{k''}{k} = 3$).  
\emph{It does not follow} that $x=y$, as opposed to what happens in the
abelian case; one can construct examples by taking $G$ to be not 
connected; I do not know what happens when $G$ is connected.]
\end{quotation}

Suppose the group $G$ in Serre's remark is a connected linear group, and
let $A$ and $B$ denote the twists of $G$ corresponding to $x$ and~$y$.
Let $X = \Isom(A,B)$, so that $X$ is a $G$-torsor.  Since $x$ and $y$
have the same images in $H^1(k',G)$ and in $H^1(k'',G)$, the variety $X$
has rational points over $k'$ and $k''$, so it has $k$-rational 
zero-cycles of degrees $\indexin{k'}{k}$ and $\indexin{k''}{k}$.  Since
these field degrees are coprime to one another, $X$ has a zero-cycle
of degree~$1$.  Thus, in this case, Totaro's question (``Does $X$ have
a rational point?'') is equivalent to Serre's implied question (``Does 
$x = y$?'').

As Serre mentions, examples can be constructed when $G$ is not 
connected.  In this paper we considered the case of the algebraic
group $G = \Aut C$ for a curve~$C$; when $C$ has genus at least $1$,
this group is not connected when it is nontrivial.





\end{document}